\RequirePackage{fix-cm}
\documentclass[smallextended]{svjour3}       
\smartqed  

\smartqed  
\usepackage{amsmath}
\usepackage{amsfonts}
\usepackage{dsfont}
\usepackage{mathrsfs}
\usepackage{latexsym}
\usepackage{graphicx}
\usepackage{amssymb}
\usepackage{epsfig,epsf,psfrag}
\usepackage{epstopdf}
\usepackage{subfigure}
\usepackage{array}
\usepackage{bm}
\usepackage{bbm}
\usepackage{diagbox}
\usepackage{multirow}
\usepackage[ruled]{algorithm2e}
\usepackage[mathlines]{lineno}
\usepackage{hyperref}
\usepackage{color}     
\usepackage{ulem}
\graphicspath{{figs/}}
\newtheorem{assumption}{Assumption}

\def\dataavailability{\par\addvspace{17pt}\small\rmfamily
\trivlist\if!\ackname!\item[]\else
\item[\hskip\labelsep
{\bfseries Data Availability}]\fi}

\newenvironment{dataavailabilitys}{\begin{dataavailability}}
{\end{dataavailability}}

\titlerunning{A semi-implicit stochastic multiscale method for radiative heat transfer problem}
\begin{document}

\title{A semi-implicit stochastic multiscale method for radiative heat transfer problem in composite materials
}


\author{Shan Zhang         \and
        Yajun Wang         \and
        Xiaofei Guan
}


\institute{Shan Zhang \at
              College of Mathematics and Physics, Shanghai University of Electric Power, Shanghai, China \\
              \email{zhangs@shiep.edu.cn}           
           \and
           Yajun Wang, Xiaofei Guan  \at
           School of Mathematical Sciences, Tongji University, Shanghai, China \\
           \email{1910733@tongji.edu.cn, guanxf@tongji.edu.cn} 
}

\date{Received: date / Accepted: date}

\maketitle

\begin{abstract}
  In this paper, we propose and analyze a new semi-implicit stochastic multiscale method for the radiative heat transfer problem with additive noise fluctuation in composite materials. In the proposed method, the strong nonlinearity term induced by heat radiation is first approximated, by a semi-implicit predictor-corrected numerical scheme, for each fixed time step, resulting in a spatially random multiscale heat transfer equation. Then, the infinite-dimensional stochastic processes are modeled and truncated using a complete orthogonal system, facilitating the reduction of the model's dimensionality in the random space. The resulting low-rank random multiscale heat transfer equation is approximated and computed by using efficient spatial basis functions based multiscale method. The main advantage of the proposed method is that it separates the computational difficulty caused by the spatial multiscale properties, the high-dimensional randomness and the strong nonlinearity of the solution, so they can be overcome separately using different strategies. The convergence analysis is carried out, and the optimal rate of convergence is also obtained for the proposed semi-implicit stochastic multiscale method. Numerical experiments on several test problems for composite materials with various microstructures are also presented to gauge the efficiency and accuracy of the proposed semi-implicit stochastic multiscale method.
\end{abstract}

\section{Introduction} 
\label{section:intro} \qquad 
Thermal radiation is important in many engineering applications, particularly at high temperatures and in the boilers and furnaces used by the petrochemical industry for hydrocarbon distillation and cracking \cite{EB:IJHMT:1973,HMDS:HP:2020}. In addition, thermal radiation is also used to monitor and control the temperature of these processes through infrared cameras and pyrometers. Therefore, the research on thermal radiation has far-reaching significance for natural science and engineering practice. Composite materials have pivotal applications in thermal protection systems due to their excellent thermal and mechanical properties, and the accurate prediction of the radiation heat transfer performance of composite materials is a key issue. The radiative heat transfer model of composites usually involves multiscale information\cite{HMDS:HP:2020,HZ:JCP:2011,SAF:JSC:2008} and uncertain features\cite{LK:JSC:2015,LQ:JNU:2019,WF:JNS:2020}. Due to the incomplete recognition of composite materials, it is necessary to develop a stochastic multiscale model. It can not only capture all random features with highly efficient simulations but also greatly alleviate the computational cost and improve computational efficiency.

The radiative heat transfer problem with rapidly oscillating coefficients has many applications in energy transport and conversion in composite materials or porous media\cite{PB:MMS:2003,EB:IJHMT:1973,Y:TEJC:2003}. Standard numerical approximations lead to extremely complex calculations, so some model reductions for various complex physical systems are necessary. Enormous progress has been made to solve multiscale problems using various technologies, including numerical homogenization approaches\cite{HKP:MMS:2020,MHCL:CCP:2020}, multiscale finite element method (MsFEM)\cite{HW:JCP:1997,HWC:MC:1999}, variational multiscale method (VMM)\cite{JSDO:SIAMJUQ:2014}, heterogeneous multiscale method (HMM) \cite{OV:MMS:2018}, generalized multiscale finite element method (GMsFEM)\cite{EGH:JCP:2013,EGLP:IJMCE:2014}, and constraint energy minimizing generalized multiscale finite element method (CEM-GMsFEM)\cite{CEL:CMAME:2018,CP:MMS:2019,FCM:JCP:2020}. In this paper, we present a multiscale model reduction technique using CEM-GMsFEM. The main components of CEM-GMsFEM are the construction of local auxiliary functions for each coarse grid (via GMsFEM) and the solution of a constraint energy minimizing (CEM) problem for each oversampled domain. The CEM-GMsFEM is spatially declining, which makes the convergence related to the size of the coarse grid, so we can construct the basis functions in a very small local region, which greatly reduces the computational burden. Some multiscale methods have been successfully applied to solve radiative heat transfer problems\cite{FJC:JCP:1971,HC:MMS:2014,HCY:AMC:2015,YCWZ:NMPDE:2016}. Few studies, however, have concentrated on the uncertainty of random microstructure or random input in radiative heat transfer problems.

The main purpose of this paper is to develop a highly efficient semi-implicit stochastic multiscale method for the time-dependent nonlinear radiative heat transfer problems driven by additive noise in composite materials. The advantages of the method are verified both theoretically and numerically. We propose and analyze a semi-implicit predictor-corrected numerical scheme resulting in a random multiscale heat transfer equation that possesses two key ingredients. First, we employ a semi-implicit time discretization technique\cite{LS:IJNAM:2013,MELD:JCP:2008,YMHYH:CMAME:2019}, which treats the nonlinear heat source term explicitly in order to avoid solving the nonlinear equations at each fixed time step and the other terms implicitly. Secondly, the semi-implicit technique is combined with the predictor-corrector method \cite{CTV:AMC:2012,DJ:JSIAM:1963,NN:AMC:2007} of the strong nonlinearity term induced by heat radiation for each fixed time step. Then the resulting random multiscale heat transfer equation is approximated and computed by a discrete representation of infinite-dimensional noise\cite{DJLQ:JCP:2018,ZGJ:JCAM:2021,Z:HK:2000} and the efficient free energy functions based on the multiscale method (CEM-GMsFEM), which only requires solving a determination heat transfer equation with a contrast-free multiscale formulation. The Monte Carlo technique is used for sampling the probability space. Moreover, the convergence of CEM-GMsFEM can be characterized by the coarse grid size and is independent of the scale length and multiscale features of the composite materials when the oversampling domain is carefully chosen. Numerical simulations are shown to support the proposed method. We present numerical results for various heterogeneous permeability fields, such as periodic porous media and permeability fields containing overlapping inclusions and non-overlapping inclusions with high conductivity values, respectively. The numerical results show that our proposed method has good generality and is not limited to periodic microstructure.

The remainder of this paper is organized as follows. In Section 2, we present the governing equations of stochastic radiative heat transfer problems. The details of the proposed method, including the discrete representation of noise and the construction of multiscale basis functions, will be introduced. The convergence of the method will be analyzed in Section 3. A few numerical results are presented in Section 4, which validate the theoretical results of the proposed method. Finally, some conclusions and comments are given. 

\section{Multiscale analysis of stochastic radiative heat transfer problems}

\qquad In this section, model problems are given in subsection \ref{modelproblems}. The representations of noises are provided in subsection \ref{noise:discretization}. Based on CEM-GMsFEM \cite{CEL:CMAME:2018}, we  present the construction for multiscale basis functions in subsection \ref{cembasis}.

\subsection{Problem setup}
\label{modelproblems}
\qquad Let $D$ be a bounded convex polygonal domain or a specified bounded smooth domain in $\mathbb{R}^{d}\ (d=2,3)$. $(\Omega,\mathcal{F},\mathbb{P})$ is a complete probability space, and $\Omega$ is the sample space, $\mathcal{F}$ is the subspace of $\sigma$-algebra in $\Omega$, $\mathbb{P}:\mathcal{F}\rightarrow [0,1]$ is the probability measure. The following governing equations of stochastic radiative heat transfer problems are considered
\begin{equation}\label{u} \small
  \left\{
 \begin{aligned}
 &u_{t}(x,t)- \nabla \cdot \big(\kappa(x) \nabla u(x,t)\big)+\sigma(x)\left(u(x,t)\right)^{4}=f(x,t)+\dot{W}(t),\quad (x,t)\in D\times (0,T], \\
 &u(x,t)=0,\quad (x,t)\in\partial D\times (0,T],\\
 &u(x,0)=u_{0}(x), \quad x \in D,
 \end{aligned}
   \right.
\end{equation}
where $u(x, t)$ represents temperature, which is a random process. $f(x,t)$ and $u_{0}(x)$ are the heat source and the initial condition, respectively. $T> 0$ is a fixed time. $\kappa(x)$ is the thermal conductivity, and $\sigma(x)$ represents the Stefan--Boltzmann coefficient. $W(t)$ represents the infinite dimensional Brownian motion on the given probability space $(\Omega,\mathcal{F},\mathbb{P})$, and $\dot{W}(t)=\frac{dW(t)}{dt}$ denotes the noise. 
 
Throughout the rest of this paper, the $L^2$-norm and the energy norm are defined by 
\begin{equation}
\|u\|_{L^{2}(D)}^{2}:=\int_{D}|u|^{2}dx,
\end{equation}
and 
\begin{equation}
  \|u\|_{a}^{2}:=\int_{D}\kappa(x) |\nabla u|^{2}dx.
\end{equation}

From the results of \cite{HC:MMS:2014} and \cite{ZK:S:2017}, we make the following assumptions.
\begin{assumption}\label{a1}
Let thermal conductivity $\kappa(x)\in L^{\infty}(\mathbb{R}^{2})$. Assume that $\kappa(x)$ is uniformly positive and bound in $D$,
i.e., there exists positive constants $\kappa_0$ and $\kappa_1$ such that
\begin{equation}
\kappa_0\xi^{T}\xi\leq \xi^{T}\kappa(x)\xi\leq \kappa_1\xi^{T}\xi,
\end{equation}
for all $\xi\in \mathbb{R}^{2}$ and $x\in D$. In particular, $\kappa(x)$ is assumed to be highly oscillatory and of high contrast, i.e., almost every $x\in D$ satisfies $\frac{\kappa_1}{\kappa_0}\gg 1$.
\end{assumption}

\begin{assumption}\label{a2}
Assume that Stefan--Boltzmann coefficients $\sigma(x)\in L^{\infty}(\mathbb{R}^{2})$ satisfies $0<\sigma_0\leq\sigma(x)\leq \sigma_1<\infty$, where $\sigma_0$ and $\sigma_1$ are some positive constants with the large ratio $\frac{\sigma_1}{\sigma_0}$. In other words, the Stefan--Boltzmann coefficients $\sigma(x)$ are also highly heterogeneous with high-contrastness $\frac{\sigma_1}{\sigma_0}$.
\end{assumption}
 
\begin{assumption}\label{a3}
  Define the function $g(u)=\sigma(x)\left(u(x,t)\right)^{4}$ to represent Stefan-Boltzmann law, which states that emitted radiations are proportional to the 4th power of the temperature. It follows
  \begin{itemize}
    \item[(1)] There exists a constant $L<C_{p}$ such that 
    \begin{equation}
    |g(s)-g(t)|(s-t)\geq -L|s-t|^{2},\quad \forall s,t\in \mathbb{R};
    \end{equation}
    \item[(2)] There exists constants $M>0$ and $R>0$ such that
    \begin{equation}
     |g(s)-g(t)|\leq M+R|s-t|,\quad \forall s,t\in \mathbb{R};
    \end{equation}
  \end{itemize}
  Here $C_{p}$ is the constant in the Poincare inequality:
  \begin{equation}\label{assumption:cp}
   \|\nabla v\|_{L^{2}(D)}^{2}\geq C_{p}\|v\|_{L^{2}(D)}^{2},~~ v\in H_0^1(D).
  \end{equation}
  \end{assumption}

\subsection{The representation of additive noise}\label{noise:discretization}

\quad \textbf{(One-dimensional Brownian motion)} 
A one dimensional continuous time stochastic process $W(t)(t\in[0,T])$ is called a standard Brownian motion if
\begin{itemize}
  \item[(1)] $W(t)$ is almost surely continuous in $t$;
  \item[(2)] $W(0)=0$ (with probability 1);
  \item[(3)] $W(t)-W(s)$ obeys the normal distribution with mean zero and variance $t-s$;
  \item[(4)] For $0 \leq  s<t<u<v \leq T$ the increments $W(t)-W(s)$ and $W(v)-W(u)$ are independent.
\end{itemize}
It can be readily shown that $W(t)$ is Gaussian process. Then the formally first-order derivative of $W(t)$ in time $\dot{W}(t)=\frac{dW(t)}{dt}$ \textcolor{black}{is referred to as white  noise.}
\begin{figure}[htpb]
  \centering
\includegraphics[width=2.2in, height=1.8in]{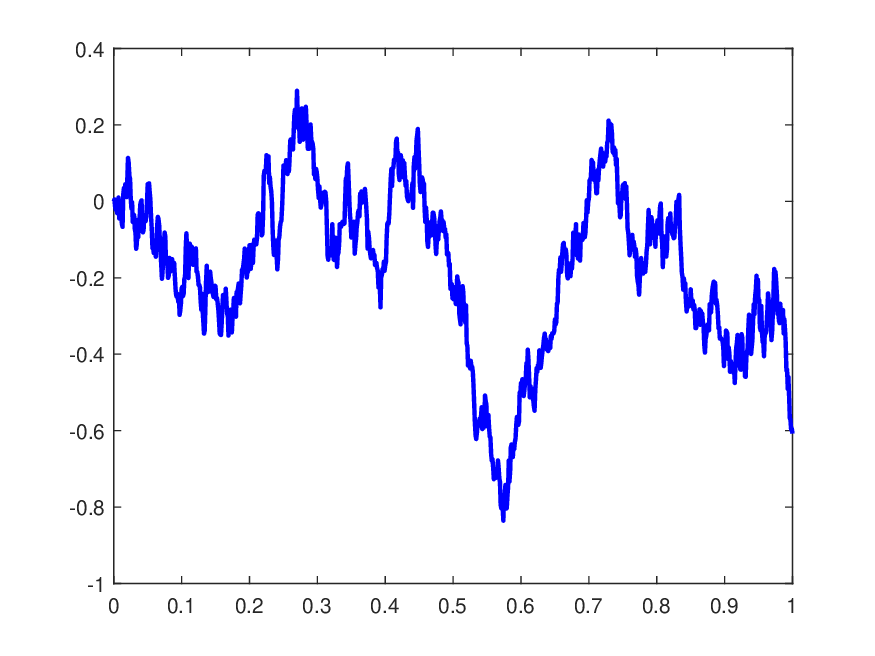}
\caption{An illustration of a sample path of Brownian motion using cumulative summation of increments}\label{white-noise}
\end{figure}
It can be seen from Fig. \ref{white-noise} that the regularity of Brownian motion is very poor. \textcolor{black}{Thus, the solution of stochastic radiative heat transfer problems driven by the white noise have poor regularity.}

\textbf{(Noises in orthogonal basis forms)} 
Brownian motion and white noise can also be defined in terms of orthogonal expansions. Suppose that $\{ m_{k}(t)\}_{k\geq 1}$ is a complete orthonormal system (CONS) in $L_{2}([0, T])$. The white noise is defined by
\begin{equation}\label{1noise}
  \dot{W}(t)=\sum_{k=1}^{\infty}\eta_{k}m_{k}(t), t \in [0, T],
  \end{equation}
  with the truncated orthogonal basis forms
  \begin{equation}\label{app1noise}
  \dot{W}_{n}(t)=\sum_{k=1}^{n}\eta_{k}m_{k}(t), t \in [0, T].
  \end{equation}
As can be seen from Fig. \ref{white-noise1}, the regularity of white noise in orthogonal basis forms has been improved, but there is still room for improvement.

  \begin{figure}[htpb]
    \subfigure[$n=4$ ]{
      \begin{minipage}[t]{0.22\linewidth}
      \centering
      \includegraphics[width=1.5in, height=1.5in]{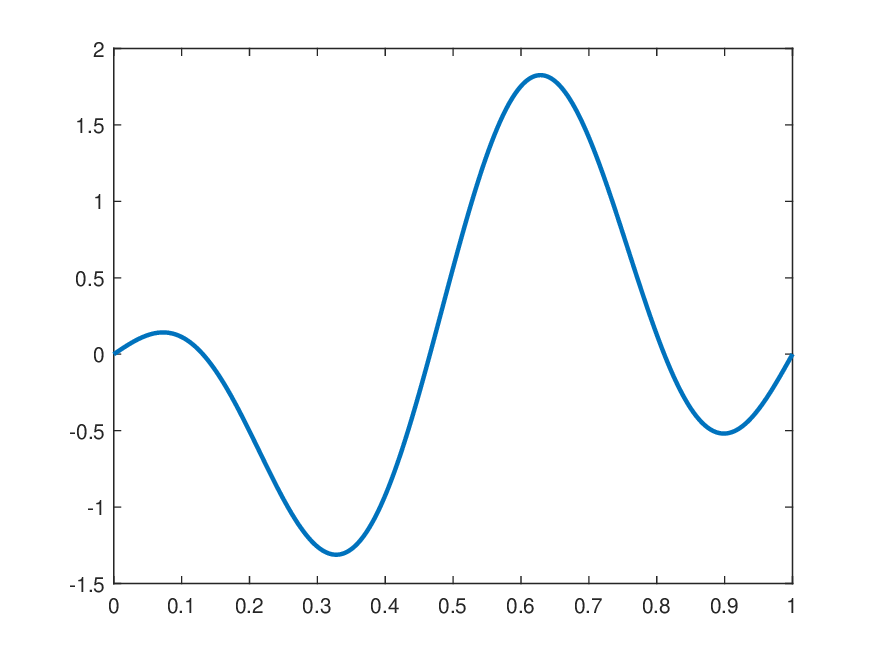}
      \end{minipage}
    }
    \subfigure[$n=8$ ]{
      \begin{minipage}[t]{0.22\linewidth}
      \centering
      \includegraphics[width=1.5in, height=1.5in]{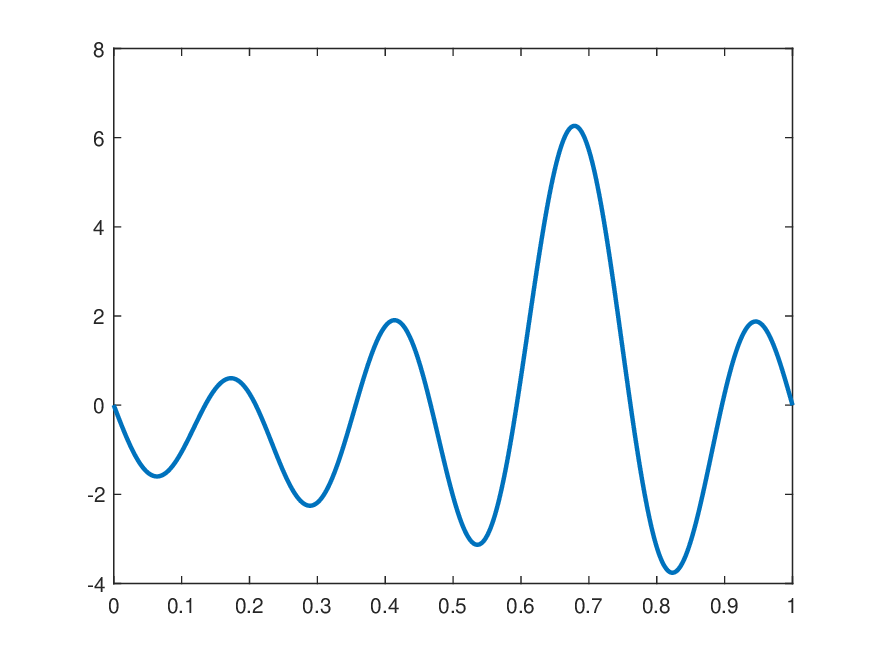}
      \end{minipage}
    }
    \subfigure[ $n=16$]{
       \begin{minipage}[t]{0.22\linewidth}
      \centering
      \includegraphics[width=1.5in, height=1.5in]{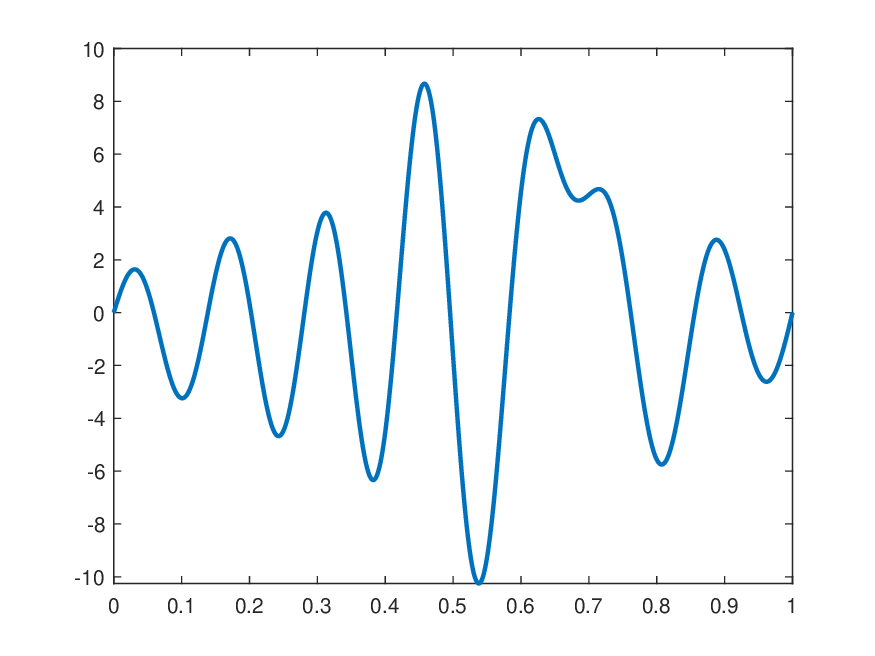}
      \end{minipage}
    }
    \subfigure[ $n=32$]{
    \begin{minipage}[t]{0.22\linewidth}
   \centering
   \includegraphics[width=1.5in, height=1.5in]{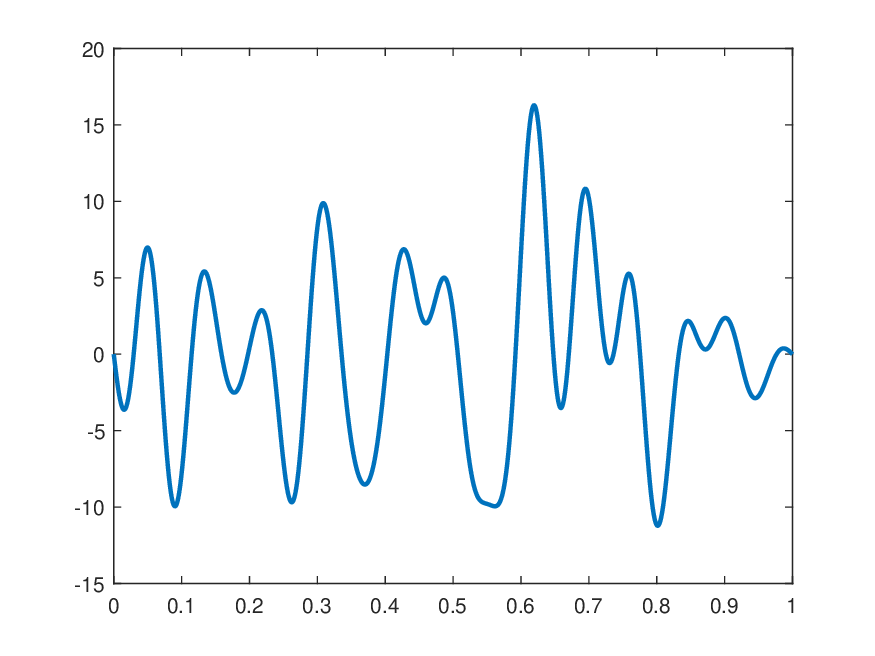}
   \end{minipage}
  }
 \caption{Truncated noise in orthogonal basis forms (\ref{1noise}) for various values of $n=4,8,16,32$. }\label{white-noise1}
  \end{figure}

\textbf{(Noises in  Fourier forms)} 
To improve the regularity of noise, following the same line \textcolor{black}{as in} \cite{DZ:SIAMJNA:2002}, the white noises can be represented by the Fourier series expansion
\begin{equation}\label{noise}
\dot{W}(t)=\sum_{k=1}^{\infty}\gamma_{k}\eta_{k}\chi_{k}(t),
\end{equation}
with the truncated expansion
\begin{equation}\label{appnoise}
\dot{W}_{n}(t)=\sum_{k=1}^{\infty}\gamma_{k}^{n}\eta_{k}\chi_{k}(t),
\end{equation}
where 
\begin{equation}
\gamma_{k}^{n}=\left\{
\begin{aligned}
\gamma & , & if\quad k\leq n, \\
0 & , & if\quad k> n.
\end{aligned}
\right.
\end{equation}
Considering the noise representation, we make the following assumptions.
\begin{itemize}
  \item[(1)] The $\{\eta_{k}\}_{k=1}^{\infty}$ are mutually independent standard Gaussian random variables and satisfy $\text{cov}(\eta_{k},\eta_{l})=E(\eta_{k}\eta_{l})=q_{kl}$, namely,
\begin{equation}\label{eq:qkl}
q_{kl}=\delta_{kl}=\left\{
\begin{aligned}
1 & , & if\quad k=l, \\
0 & , & if\quad k\neq l.
\end{aligned}
\right.
\end{equation}
  \item[(2)] The deterministic functions $\{\chi_{k}(t)\}_{k=1}^{\infty}$ form an orthonormal basis in $L^{2}([0,T])$.
  \item[(3)] The coefficients $\{\gamma_{k}\}_{k=1}^{\infty}$ are to be selected to ensure the convergence of the series. For example, we can choose different decay coefficients: $\gamma_{k}=\frac{1}{2^{k}}, \frac{1}{k^{3/2}}$, and so on.
\end{itemize}

\begin{figure}[htpb]
  \subfigure[$n=4$]{
    \begin{minipage}[t]{0.22\linewidth}
    \centering
    \includegraphics[width=1.5in, height=1.5in]{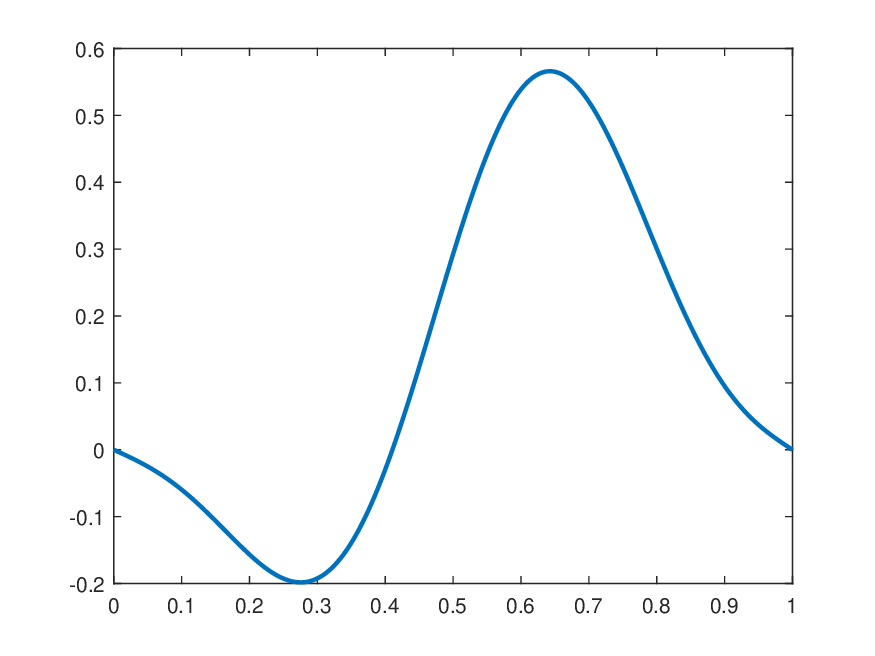}
    \end{minipage}
  }
  \subfigure[$n=8$]{
    \begin{minipage}[t]{0.22\linewidth}
    \centering
    \includegraphics[width=1.5in, height=1.5in]{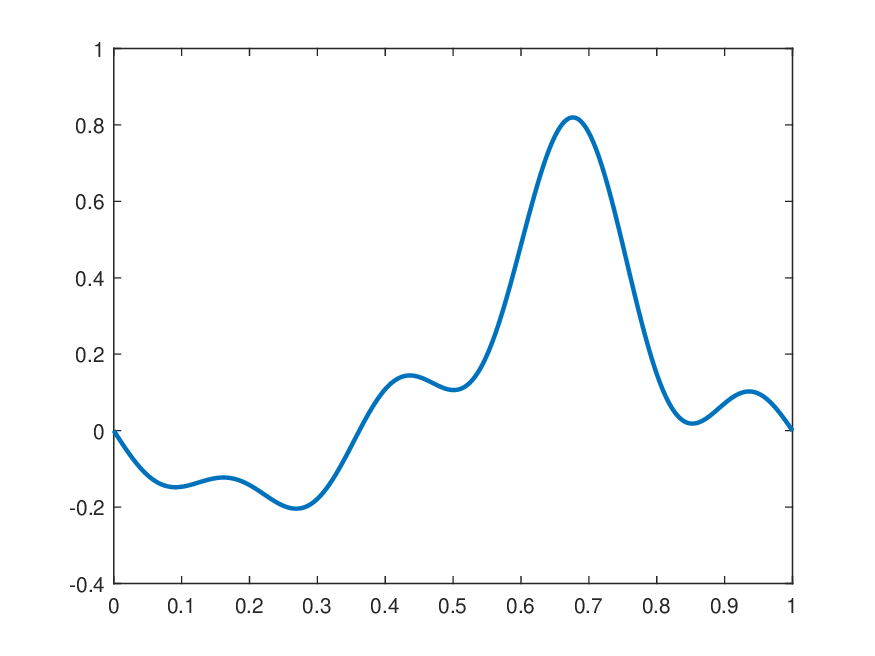}
    \end{minipage}
  }
  \subfigure[$n=16$]{
     \begin{minipage}[t]{0.22\linewidth}
    \centering
    \includegraphics[width=1.5in, height=1.5in]{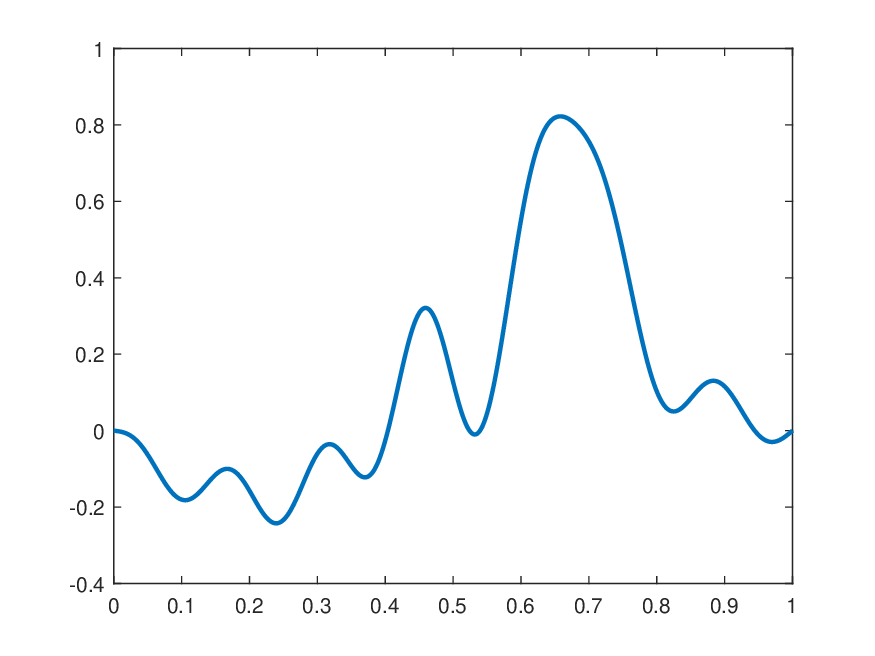}
    \end{minipage}
  }
  \subfigure[$n=32$]{
  \begin{minipage}[t]{0.22\linewidth}
 \centering
 \includegraphics[width=1.5in, height=1.5in]{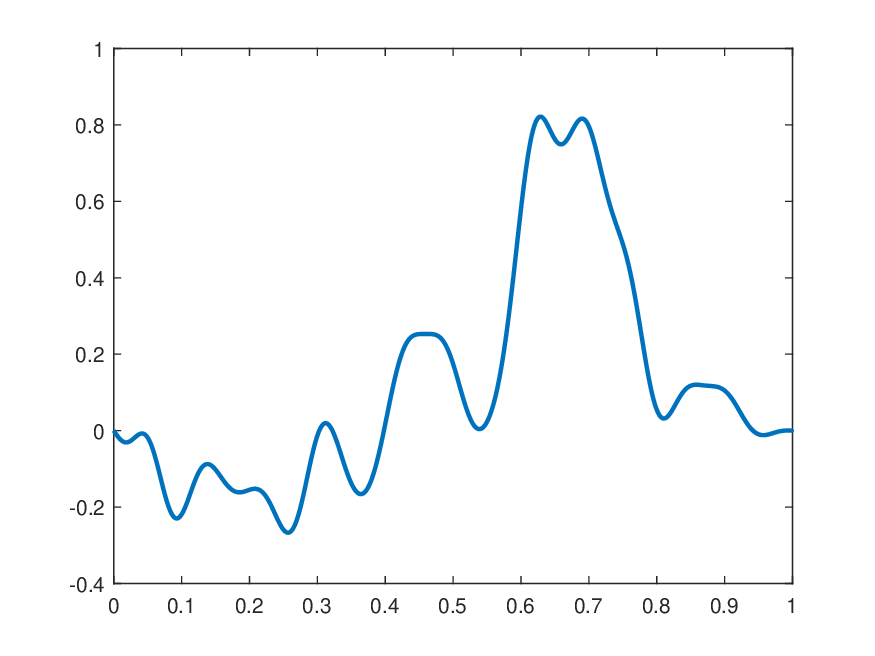}
 \end{minipage}
}
\caption{Truncated noise in  Fourier forms (\ref{appnoise}) with decay coefficients $\gamma_{k}=\frac{1}{k^{3/2}}$ for various values of $n=4,8,16,32$. }\label{Fourierwhite-noise1}
\end{figure}
\begin{figure}[htpb]
  \subfigure[$n=4$]{
    \begin{minipage}[t]{0.22\linewidth}
    \centering
    \includegraphics[width=1.5in, height=1.5in]{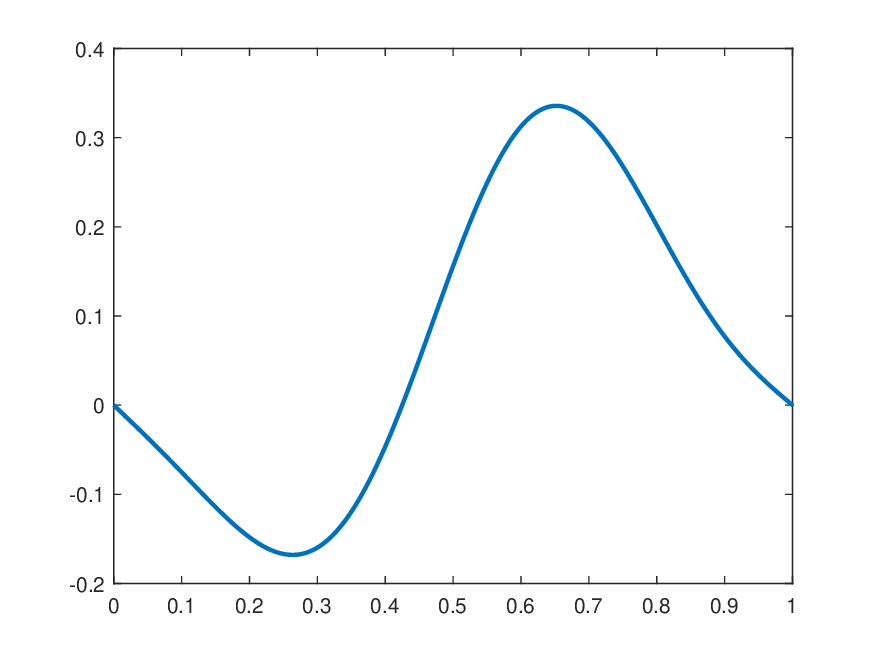}
    \end{minipage}
  }
  \subfigure[$n=8$]{
    \begin{minipage}[t]{0.22\linewidth}
    \centering
    \includegraphics[width=1.5in, height=1.5in]{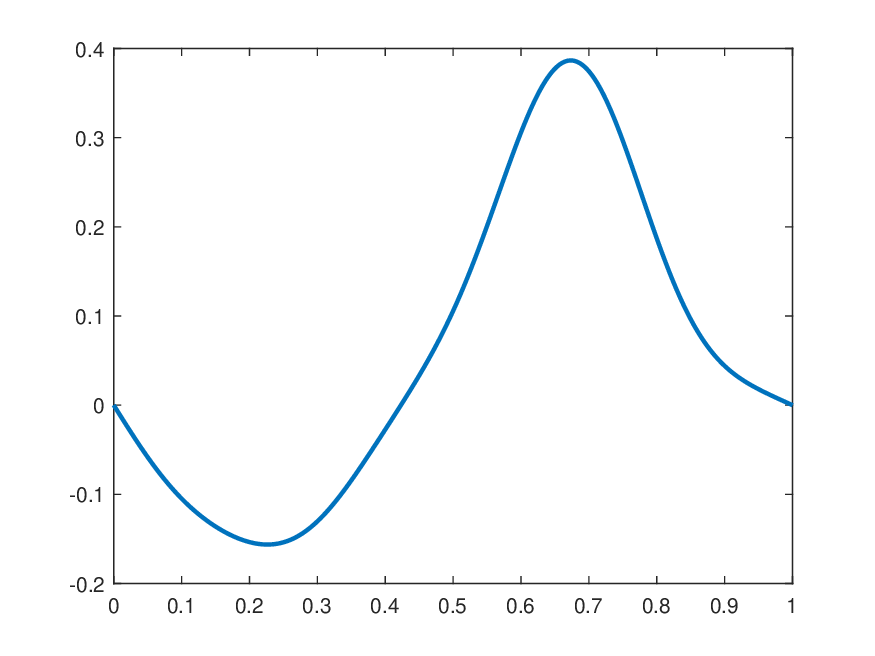}
    \end{minipage}
  }
  \subfigure[$n=16$]{
     \begin{minipage}[t]{0.22\linewidth}
    \centering
    \includegraphics[width=1.5in, height=1.5in]{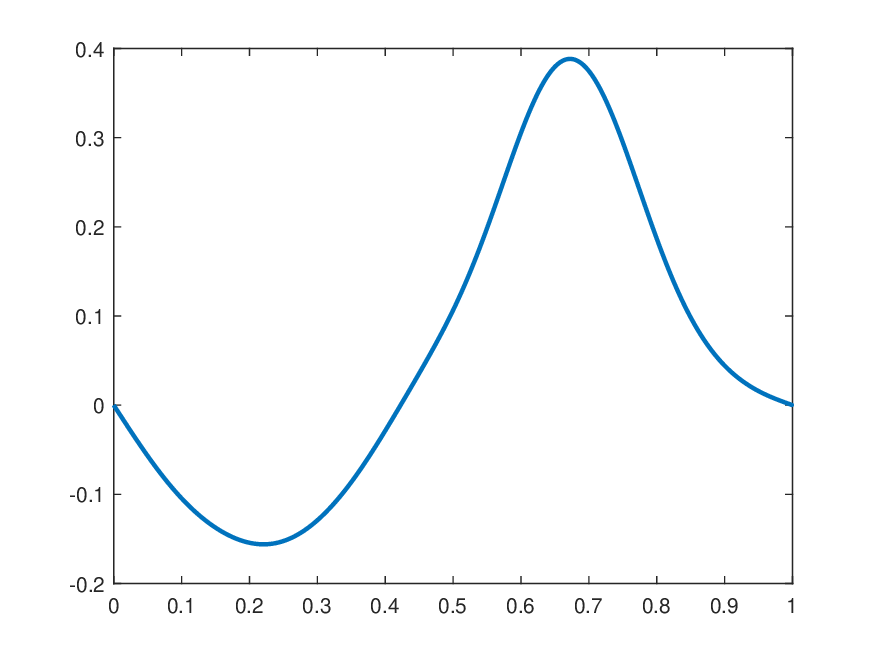}
    \end{minipage}
  }
  \subfigure[$n=32$]{
  \begin{minipage}[t]{0.22\linewidth}
 \centering
 \includegraphics[width=1.5in, height=1.5in]{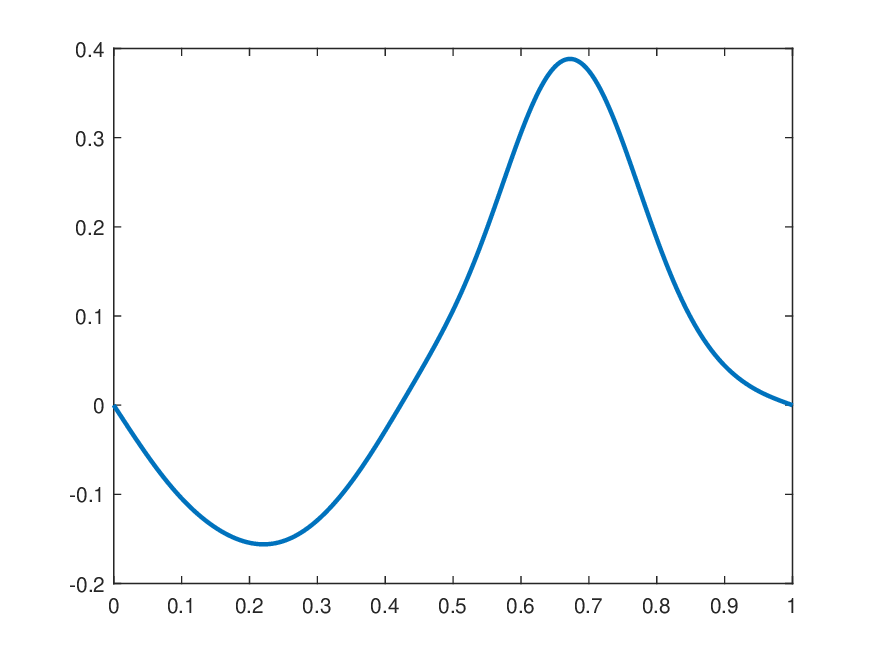}
 \end{minipage}
}
\caption{Truncated noise in  Fourier forms (\ref{appnoise}) with decay coefficients $\gamma_{k}=\frac{1}{2^{k}}$ for various values of $n=4,8,16,32$. }\label{Fourierwhite-noise2}
\end{figure}
Fig. \ref{Fourierwhite-noise1} and Fig. \ref{Fourierwhite-noise2} provide sample realizations of noises having forms ({\ref{appnoise}}) in the Fourier basis with coefficients $\gamma_{k}=\frac{1}{k^{3/2}}$ and $\gamma_{k}=\frac{1}{2^{k}}$ for various values of $n=4,8,16,32$, respectively. We found that the faster the coefficient decays, the smoother the image of truncated noise ({\ref{appnoise}}). The approximation of truncated noise becomes more accurate as the number of noise truncations $n$ increases.

\textcolor{black}{We now substitute $\dot{W}(t)$ by the finite dimensional noise $\dot{W}_{n}(t)$ in (\ref{u}). Thus, $u_{n}(x,t)$ satisfies the following equation}
\begin{equation}\label{equ:un} \small
  \left\{
 \begin{aligned}
 (u_{n}(x,t))_{t}- \nabla \cdot \big(\kappa(x,t) \nabla \big)u_{n}(x,t)+g(u_{n}(x,t))&=f(x,t)+\dot{W}_{n}(t),\quad (x,t)\in D\times (0,T], \\
 u_{n}(x,t)&=0,\quad (x,t)\in\partial D\times (0,T],\\
 u_{n}(x,0)&=u_{0}(x),\quad x\in D.
 \end{aligned}
   \right.
 \end{equation}

\subsection{Formulae of stochastic CEM-GMsFEM}
\label{cembasis}
\qquad The domain $D$ (Fig. \ref{grid1}) is composed of a family of meshes $\mathcal{T}^{H}$, where $H=\underset{K_i\in\mathcal{T}^{H}}{\max}H_{K_i}$ is the mesh size. $\mathcal{T}^{H}$ is the conforming partition of $D$, and is shape regular. $\mathcal{T}^{h}$ is a conforming refinement of $\mathcal{T}^{H}$, and $h$ is the diameter of fine grid. $N$ denotes the number of elements in $\mathcal{T}^{H}$, and $N_{c}$ denotes the number of vertices of all coarse grid. Let $\{x_{i}\}_{i=1}^{N_{c}}$ be the set of vertices in $\mathcal{T}^{H}$ and $\omega_{i}=\bigcup\{K_{j}\in \mathcal{T}^{H}|x_{i}\in\overline{K_{j}}\}$ be the neighborhood of the node $x_{i}$. For each coarse block $K_i$, the oversampling region $K_{i,m}\subset D$ is defined by enlarging $K_{i}$ with several coarse grid layers.
\begin{figure}[ht]
  \centering
  \includegraphics[width=2.0in, height=1.8in]{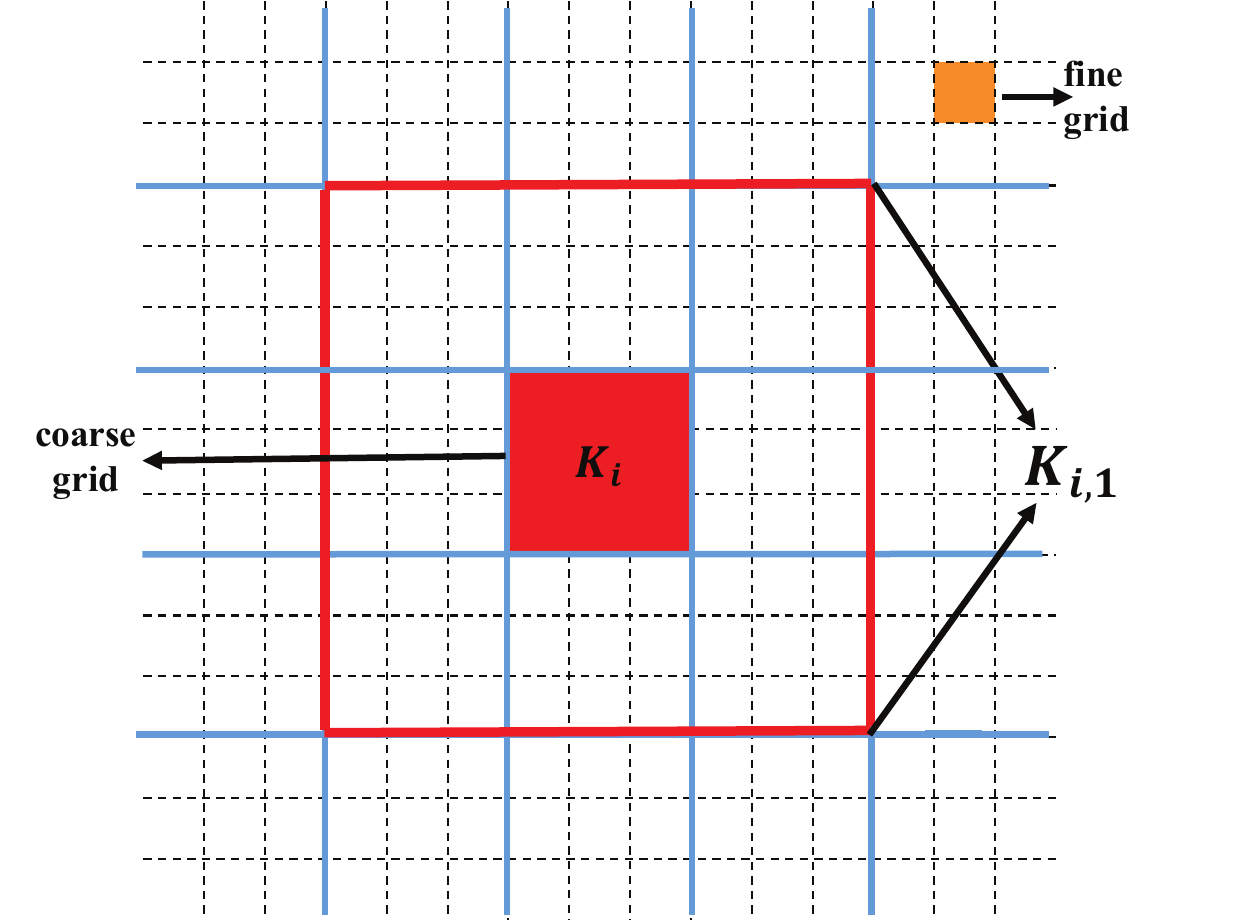}
  \includegraphics[width=1.6in, height=1.8in]{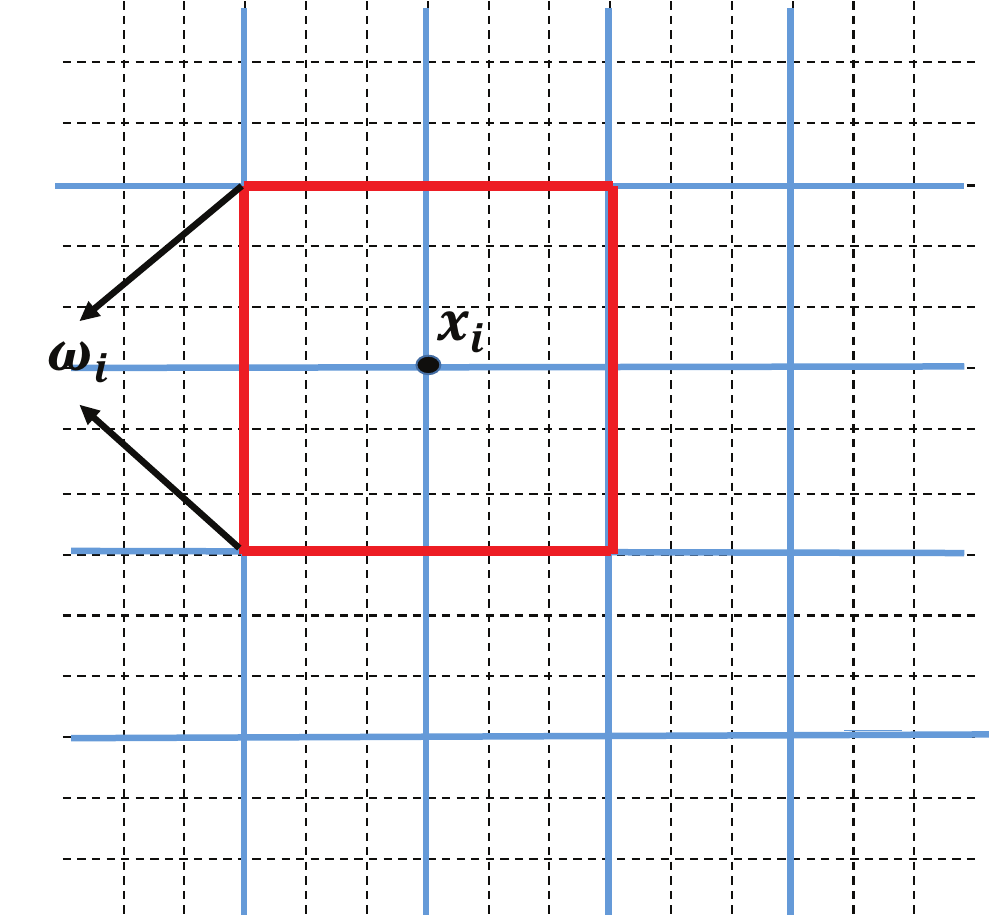}
  \caption{ The fine grid, coarse grid $K_{i}$, oversampling domain $K_{i,1}$, and neighborhood $\omega_{i}$ of the nodes $x_{i}$.}\label{grid1}
 \end{figure}
 Then, the procedure for the construction of the stochastic CEM-GMsFEM formulation can be divided into three steps. Firstly, let $V(K_{i})=H^{1}(K_{i})$, the following spectral problem is given to find eigen-pairs $\{\lambda_{j}^{(i)},  \phi_{j}^{(i)}\}\in \mathbb{R}\times V(K_{i})$ such that
\begin{equation}\label{cem-auxproblem}
a_{i}(\phi_{j}^{(i)},v) =\lambda_{j}^{(i)}s_{i}(\phi_{j}^{(i)},v)\quad \forall v\in V(K_{i}),
\end{equation}
where
\begin{equation}
a_{i}(u,v) = \int_{K_{i}} \kappa \nabla u \cdot \nabla v dx, \quad s_{i}(u,v)=\int_{K_{i}}\widetilde{\kappa}uv dx.
\end{equation}
Here $\widetilde{\kappa}=\kappa\sum_{j=1}^{N_{c}}|\nabla\chi_{j}|^{2}$ and $\{\chi_{j}\}$ is a set of partition of unity functions for the coarse partition. We assume that the eigenfunctions satisfy the normalized condition $s_{i}(\phi_{j}^{(i)},\phi_{j}^{(i)})=1$. Let the eigenvalues $\lambda_{j}^{(i)}$ be arranged in ascending order, i.e., $\lambda_{1}^{(i)}\leq \lambda_{2}^{(i)}\leq \cdots$. We define the local auxiliary multiscale space $V_{aux}^{(i)}$ by using the first $L_{i}$ eigenfunctions
\begin{equation}
V_{aux}^{(i)}=\text{span} \{\phi_{j}^{(i)}|1\leq j\leq L_{i}\}.
\end{equation}
Then the global auxiliary space $V_{aux}$ is defined by using these local auxiliary spaces, i.e.,
\begin{equation}
V_{aux}=\bigoplus_{i=1}^{N}V_{aux}^{(i)}.
\end{equation}
To construct the CEM-GMsFEM basis functions, we need the following definition.
\begin{definition}[$\phi_{j}^{(i)}$-orthogonality]
  Given a function $\phi_{j}^{(i)}\in V_{aux}$, if a function $\psi\in V$ satisfies
\begin{equation}
s(\psi,\phi_{j}^{(i)}) = 1,\quad s(\psi,\phi_{j'}^{(i')}) = 0\quad if \ j'\neq j \ or \  i'\neq i,
\end{equation}
then we say that $\psi$ is $\phi_{j}^{(i)}$-orthogonal, where $\displaystyle s(u,v)=\sum_{i=1}^{N}s_{i}(u,v)$. 
\end{definition}
We define an operator $\pi:V\rightarrow V_{aux}$ by
\begin{equation}
\pi(v) = \sum_{i=1}^{N}\sum_{j=1}^{L_{i}}s_{i}(v,\phi_{j}^{(i)})\phi_{j}^{(i)},\quad\forall v\in V.
\end{equation}
The null space of the operator $\pi$ is defined by $\widetilde{V}=\{v\in V|\pi(v)=0\}$.

Secondly, let $H^1_{0}(D)(K_{i,m}):=H_0^1(K_{i,m})$, the multiscale basis function $\psi_{j,ms}^{(i)}\in H^1_{0}(D)(K_{i,m})$ can be obtained by solving the following constrained energy minimization problem
\begin{equation}\label{min}
\psi_{j,ms}^{(i)}=\text{argmin}\big\{a(\psi,\psi)|\psi\in H^1_{0}(D)(K_{i,m}), \psi \ is \ \phi_{j}^{(i)}\text{-orthogonal}\big\}.
\end{equation}
The multiscale spaces $V_{ms}$ can be given by all multiscale basis functions
$\psi_{j,ms}^{(i)}$ as follows
\begin{equation}
  V_{ms}=\text{span}\big\{\psi_{j,ms}^{(i)}|1\leq j\leq L_{i},1\leq i\leq N\big\}.
\end{equation}

\subsection{Semi-implicit predictor-corrected numerical scheme with spatial discretization methods}
\label{algrithom:computation}
\qquad In this subsection, we present a semi-implicit predictor-corrected numerical scheme with the standard spatial finite element method(\text{FEM}) and CEM-GMsFEM, respectively. The discrete forms of the solution $u$ satisfies the Eq.(\ref{eq-uweak}) and multiscale solution $u_{n}^{ms}$ satisfies the Eq.(\ref{eq-umsweak}). Let $\left\{t^{j}\right\}_{j=1}^I $ be a uniform partition in the time direction with $t^{j}:=j\Delta t$, $T=I\Delta t$ and denote
\begin{equation}
 u^{j}:=u(x,t^{j}), \quad u_{n}^{h,j}:=u_{n}^{h}(x,t^{j}),\quad u_{n}^{ms,j}:=u_{n}^{ms}(x,t^{j}),
\end{equation}
\begin{equation}
 f^j = f(x,t_j), \quad \dot{W}^j=\dot{W}(t_j), \quad \dot{W}_{n}^j=\dot{W}_{n}(t_j),
\end{equation}
for $j=1,2,\cdots, I$.

Now we are ready to introduce the semi-implicit predictor-corrected numerical scheme with Galerkin finite element approximations for problem (\ref{u}). For $j=1,2,\cdots, I$, find $u_{n}^{h,j}\in V_{h}$ such that 
\begin{equation}\label{fem-discrete} \small
  \left\{
   \begin{aligned}
  \bigg(\frac{\hat{u_{n}}^{h,j}-u_{n}^{h,j-1}}{\Delta t},v\bigg)+a(\hat{u_{n}}^{h,j},v)+\sigma(x)(u_{n}^{h,j-1})^3(\hat{u_{n}}^{h,j} ,v)&=(f^j,v)+(\dot{W}_{n}^j,v),\\
  \bigg(\frac{u_{n}^{h,j}-u_{n}^{h,j-1}}{\Delta t},v\bigg)+a(u_{n}^{h,j},v)+\sigma(x)(\hat{u_{n}}^{h,j})^3(u_{n}^{h,j} ,v)&=(f^j,v)+(\dot{W}_{n}^j,v),\\  
    (u_{n}^{h,0},v)&=(u_{0},v),
  \end{aligned}
     \right.
  \end{equation}
$\forall ~ v\in V_{ms}.$  By using the expression $u_{n}^{h,j}:=\Sigma_{l=1}^{N_{f}}u_{l}(t_{j})\varphi_{l}^{fem}(x)$, where $N_{f}$ is the number of degree of freedoms in the \text{FEM} and $\varphi_{l}^{fem}(x)$ denotes finite element basis functions corresponding to fine grids, it can lead to a system of algebraic equations:
 \begin{equation}\label{fem-matric}
 \left\{
 \begin{aligned}
\bm{M}\bm{\hat{U} }^{j}+\Delta t \bm{A}\bm{\hat{U} }^{j}+\Delta t\bm{N}(\bm{U}^{j-1})\bm{\hat{U} }^{j}&=\Delta t \bm{F}^{j}+\Delta t\bm{\Upsilon}^{j}+\bm{M}\bm{U}^{j-1},\\
\bm{M}\bm{U}^{j}+\Delta t \bm{A}\bm{U}^{j}+\Delta t\bm{N}(\bm{\hat{U} }^{j})\bm{U}^{j}&=\Delta t \bm{F}^{j}+\Delta t\bm{\Upsilon}^{j}+\bm{M}\bm{U}^{j-1},\\
 \bm{U}^{0}&=\bm{U}(0),
    \end{aligned}
   \right.
\end{equation}
where $\Delta t$ is the time-step size and the superscript $j$ represents the temporal level of the solution. \textcolor{black}{If we set} $\bm{M}=(\bm{m}_{lq})\in \mathds{R}^{N_{f}\times N_{f}}, \bm{A}=(\bm{\alpha}_{lq})\in \mathds{R}^{N_{f}\times N_{f}}, \bm{N}(z)=\bm{n}_{lq}(z)\in \mathds{R}^{N_{f}\times N_{f}},\bm{F}=(\bm{\beta}_{q})\in \mathds{R}^{N_{f}},
\bm{\Upsilon}=(\bm{\gamma}_{q})\in \mathds{R}^{N_{f}},\bm{U^{j}}\in \mathds{R}^{N_{f}}$, and
\begin{equation}
 \begin{split}
 \bm{m}_{lq}&=\int_{D}\varphi_{l}^{fem}(x)\varphi_{q}^{fem}(x),\\
  \bm{\alpha}_{lq}&=\int_{D}\kappa(x) \nabla\varphi_{l}^{fem}(x)\cdot \nabla\varphi_{q}^{fem}(x),\\
  \bm{n}_{lq}(z)&=\sigma(x)z^{3}\int_{D} \varphi_{l}^{fem}(x)\varphi_{q}^{fem}(x),\\
 \bm{\beta}_{q}&=\int_{D}f(x,t_{j})\varphi_{q}^{fem}(x),\\
  \bm{\gamma}_{q}&=\sum_{k=1}^{\infty}\gamma_{k}\eta_{k}\int_{D}\psi_{k}(t_{j})\varphi_{q}^{fem}(x),\\
  \bm{U}^{j}&=[u_{1}(t_{j}),u_{2}(t_{j}),\cdots,u_{N_{f}}(t_{j})]^{T},\\
 \end{split}
\end{equation}
for $l=1,2,\cdots,N_{f}$, $q=1,2,\cdots,N_{f}$.

Based on the subsection \ref{cembasis}, we use the multiscale basis functions to approximate the solution. The semi-implicit predictor-corrected numerical scheme based on CEM-GMsFEM is
\begin{equation}\label{cem-discrete} \small
\left\{
 \begin{aligned}
\bigg(\frac{\hat{u_{n}}^{ms,j}-u_{n}^{ms,j-1}}{\Delta t},v\bigg)+a(\hat{u_{n}}^{ms,j},v)+\sigma(x)(u_{n}^{ms,j-1})^3(\hat{u_{n}}^{ms,j} ,v)&=(f^j,v)+(\dot{W}_{n}^j,v),\\
\bigg(\frac{u_{n}^{ms,j}-u_{n}^{ms,j-1}}{\Delta t},v\bigg)+a(u_{n}^{ms,j},v)+\sigma(x)(\hat{u_{n}}^{ms,j})^3(u_{n}^{ms,j} ,v)&=(f^j,v)+(\dot{W}_{n}^j,v),\\  
  (u_{n}^{ms,0},v)&=(u_{0,ms},v),
\end{aligned}
   \right.
\end{equation}
$\forall~ v\in V_{ms}.$ Then, the multiscale solutions can be represented by $u_{n}^{ms,j} :=\Sigma_{l=1}^{N_{ms}}y_{l}(t_{j})\phi_{l}^{ms}(x)$, where $N_{ms}$ is the number of degree of freedoms in the \text{CEM-GMsFEM}, and $\phi_{l}^{ms}(x)$ denotes multiscale basis functions corresponding to course grids. In fact, the above multiscale equations (\ref{cem-discrete}) can be written in matrix form as
 \begin{equation}\label{cem-matric}
 \left\{
 \begin{aligned}
\bm{\widehat{M}}\bm{\hat{Y} }^{j}+\Delta t \bm{\widehat{A}}\bm{\hat{Y} }^{j}+\Delta t\bm{\widehat{N}}(\bm{Y}^{j-1})\bm{\hat{Y} }^{j}&=\Delta t \bm{\widehat{F}}^{j}+\Delta t\bm{\widehat{\Upsilon}}^{j}+\bm{\widehat{M}}\bm{Y}^{j-1},\\
\bm{\widehat{M}}\bm{Y}^{j}+\Delta t \bm{\widehat{A}}\bm{Y}^{j}+\Delta t\bm{\widehat{N}}(\bm{\hat{Y} }^{j})\bm{Y}^{j}&=\Delta t \bm{\widehat{F}}^{j}+\Delta t\bm{\widehat{\Upsilon}}^{j}+\bm{\widehat{M}}\bm{Y}^{j-1},\\
\bm{Y}^{0}&=\bm{Y}(0),
    \end{aligned}
   \right.
\end{equation}
where we set $\bm{\widehat{F}}=(\bm{\widehat{\beta}}_{q})\in \mathds{R}^{N_{ms}},
\bm{\widehat{\Upsilon}}=(\bm{\widehat{\gamma}}_{q})\in \mathds{R}^{N_{ms}},\bm{Y^{j}}\in \mathds{R}^{N_{ms}}$, and
\begin{equation}
 \begin{split}
 \bm{\widehat{m}}_{lq}&=\int_{D}\phi_{l}^{ms}(x)\phi_{q}^{ms}(x),\\
  \bm{\widehat{\alpha}}_{lq}&=\int_{D}\kappa(x) \nabla\phi_{l}^{ms}(x)\cdot \nabla\phi_{q}^{ms}(x),\\
  \bm{\widehat{n}}_{lq}(z)&=\sigma(x)(z)^{3}\int_{D} \phi_{l}^{ms}(x)\phi_{q}^{ms}(x),\\
 \bm{\widehat{\beta}}_{q}&=\int_{D}f(x,t_{j})\phi_{q}^{ms}(x),\\
  \bm{\widehat{\gamma}}_{q}&=\sum_{k=1}^{\infty}\gamma_{k}^{n}\eta_{k}\int_{D}\psi_{k}(t_{j})\phi_{q}^{ms}(x),\\
  \bm{Y}^{j}&=[y_{1}(t_{j}),y_{2}(t_{j}),\cdots,y_{N_{ms}}(t_{j})]^{T},\\
 \end{split}
\end{equation}
for $l=1,2,\cdots,N_{ms}$, $q=1,2,\cdots,N_{ms}$.

Set $\bm{{\phi}^{ms}}=[\phi_{1}^{ms},\phi_{2}^{ms}, \cdots,\phi_{N_{ms}}^{ms}]$, then we can find
\begin{equation}
 \begin{split}
 \bm{\widehat{M}}&=(\bm{\widehat{m}}_{lq})=(\bm{{\phi}^{ms}})^{T}\bm{M}\bm{{\phi}^{ms}}\in \mathds{R}^{N_{ms}\times N_{ms}}, \\ \bm{\widehat{A}}&=(\bm{\widehat{\alpha}}_{lq})=(\bm{{\phi}^{ms}})^{T}\bm{A}\bm{{\phi}^{ms}}\in \mathds{R}^{N_{ms}\times N_{ms}}, \\ \bm{\widehat{N}}&=(\bm{\widehat{n}}_{lq})=(\bm{{\phi}^{ms}})^{T}\bm{N}\bm{{\phi}^{ms}}\in \mathds{R}^{N_{ms}\times N_{ms}},\\
  \end{split}
\end{equation}
where $N_{ms}\ll N_{f}$.
\textcolor{black}{If we choose a suitable linear system solver, the computational cost of Eq.(\ref{cem-discrete}) is $O(N_{ms}^2)$, while The computational cost of Eq. (\ref{fem-discrete}) is $O(N_{f}^2)$.
}
This implies that the reduced model Eq. (\ref{cem-discrete}) can lessen the computational burden of the full model (\ref{fem-discrete}), i.e., instead of direct numerical simulations on fine grids, CEM-GMsFEM is applied as a model reduction technique. It can solve these multiscale model problems in a lower dimensional space to improve computation efficiency.

\section{The convergence analysis}
\qquad \textcolor{black}{In this section, we analyze the convergence and error estimates for the proposed semi-implicit stochastic multiscale method of the problem (\ref{u}).} The error can be divided into two parts: the approximate truncated error estimate $\mathbb{E} \|u-u_{n}\|_{L^{2}(D)}^{2}$ and the multiscale error estimate $\mathbb{E} \|u_{n}-u_{ms}\|_{L^{2}(D)}^{2}$. In the subsection \ref{The convergence analysis:u-un}, we analyze the error estimate between the approximate solution $u_{n}$ and the solution of the original equation $u$ in the form of the mild solution. In the subsection \ref{The convergence analysis:un-ums}, we derive the error estimates for multiscale methods using the variational solutions. 
\subsection{Approximate truncated error estimate}
\label{The convergence analysis:u-un}
\begin{definition}(Mild Solution)
 When $A=-\nabla \cdot \big(\kappa(x) \nabla \big)$ is a second-order differential and positive definite operator, the mild solution for Eq. (\ref{u}) is
  \begin{equation}\label{Equ:umild} \small
  u(t)=e^{-tA}u_{0}+\int_0^t e^{-(t-s)A}f(s)ds+\int_0^t e^{-(t-s)A}dW(s)-\int_0^t e^{-(t-s)A}g(u(s))ds, \quad \forall t\in (0,T],
  \end{equation}
where $e^{-tA}(t\in (0,T])$ is the analytic $C_{0}-$semigroup generalized by $A$.
\end{definition}

Similarly, after approximating the infinite-dimensional noise $\dot{W}(t)$ by truncation noise $\dot{W}_{n}(t)$, the mild solution $u_{n}(t)$ of problems (\ref{equ:un}) can be obtained based on $It\hat{o}$ integration
\begin{equation}\label{Equ:unmild} \small
u_{n}(t)=e^{-tA}u_{0}+\int_0^t e^{-(t-s)A}f(s)ds+\int_0^t e^{-(t-s)A}dW_{n}(s)-\int_0^t e^{-(t-s)A}g(u_{n}(s))ds, \quad \forall t\in (0,T].
\end{equation}

Under the mild solutions (\ref{Equ:umild}) and the results of \cite{ZGJ:JCAM:2021}, we can conclude
\begin{lemma}\label{thmzeta}
Let $\displaystyle\zeta(t)=\int_0^t e^{-(t-s)A}dW(s)-\int_0^t e^{-(t-s)A}dW_{n}(s)$, there hold
\begin{equation}\displaystyle\mathbb{E} \|\zeta(t)\|_{L^{2}(D)}^{2} \leq C \|\bm{\gamma}-\bm{\gamma^{n}}\|_{Q_{-1}}^{2},\quad \forall t> 0,
\end{equation}
where
\begin{equation}
\bm{\gamma^{n}}=(\gamma_{1}^{n},\gamma_{2}^{n},...,\gamma_{k}^{n},...)^{T},
\end{equation}
and 
\begin{equation}
\bm{\gamma}=(\gamma_{1},\gamma_{2},...,\gamma_{k},...)^{T},
\end{equation}
are infinite column vectors of noise's coefficients.\\
The seminorm is defined by
\begin{equation}
 \|\bm{\gamma}\|_{Q_{s}}^{2}=(\bm{\gamma},\bm{\gamma})_{Q_{s}}=\bm{\gamma}^{T}\cdot Q_{s} \cdot \bm{\gamma}=\sum_{k=1}^{\infty}\sum_{l=1}^{\infty}\gamma_{k}\gamma_{l}(kl)^{s}q_{kl},
\end{equation}
where $Q_{s}$ represents the infinite matrix with entries $Q_{s}=((kl)^{s}q_{kl})_{k,l=1}^{\infty}$ for an integer $s$, and $q_{kl}$ is defined by Eq.(\ref{eq:qkl}).
\end{lemma}
Lemma \ref{thmzeta} implies that $\mathbb{E}\|\zeta(t)\|_{L^{2}(D)}^{2}$ converges to $0$ if $\bm{\gamma}^n$ converges to $\bm{\gamma}$.

\begin{theorem}\label{thmapproerror}
if $u$ and $u_{n}$ are the solutions of (\ref{Equ:umild}) and (\ref{Equ:unmild}), respectively. Then, we can conclude that
\begin{equation}
  \mathbb{E} \|u-u_{n}\|_{L^{2}(D)}^{2} \leq C(M\mathbb{E}\|\zeta\|_{L^{2}(D)}+(C_{p}+R)^{2}\mathbb{E}\|\zeta\|_{L^{2}(D)}^{2}),
\end{equation}
where $C_{p}$ is defined in \ref{assumption:cp} and the constant $C$ depends only on $C_{p}$, $L$ and $M$, $R$.
\begin{proof}
  Subtracting Eq.(\ref{Equ:unmild}) from Eq.(\ref{Equ:umild}), we get
  \begin{equation}\label{Equmild:u-un}
  u(t)-u_{n}(t)=\zeta(t)-\int_0^t e^{-(t-s)A}[g(u(s)-g(u_{n}(s))]ds.
  \end{equation}
  Multiplying both sides of Eq.(\ref{Equmild:u-un}) by $g(u(t))-g(u_{n}(t))$, integrating over the region $D$ and using the Assumption \ref{a3}, we have
  \begin{equation}
  \begin{split}
  -L\|u-u_{n}\|_{L^{2}(D)}^{2}
  &\leq \int_{D} \zeta(t)[g(u(t)-g(u_{n}(t))]dx\\
  &~~~-\int_{D} \bigg(\int_0^t e^{-(t-s)A}[g(u(s)-g(u_{n}(s))]ds[g(u(t)-g(u_{n}(t))]\bigg)dx\\
  &\leq \int_{D} \zeta(t)[g(u(t)-g(u_{n}(t))]dx-C_{p}\\
  &~~~\int_{D} \bigg(\int_0^t e^{-(t-s)A}[g(u(s)-g(u_{n}(s))]ds\bigg)^{2}dx\\
  &=\int_{D} \zeta(t)[g(u(t)-g(u_{n}(t))]dx-C_{p}\int_{D}[u-u_{n}-\zeta]^{2}dx,
  \end{split}
  \end{equation}
  i.e.,
  \begin{equation}
  -L\|u-u_{n}\|_{L^{2}(D)}^{2}\leq\int_{D} \zeta(t)[g(u(t)-g(u_{n}(t))]dx-C_{p}\int_{D}[u-u_{n}-\zeta]^{2}dx.
  \end{equation}
  Then, by Assumption \ref{a3} and Cauchy-Schwarz inequality, we have
  \begin{equation}
  (C_{p}-L)\|u-u_{n}\|_{L^{2}(D)}^{2}\leq (C_{p}+R)\|u-u_{n}\|_{L^{2}(D)}\|\zeta\|_{L^{2}(D)}+2M\|\zeta\|_{L^{2}(D)}.
  \end{equation}  
  Since $C_{p}-L\geq 0$, it implies that 
  \begin{equation}
  \|u-u_{n}\|_{L^{2}(D)}^{2}\leq C[M\|\zeta\|_{L^{2}(D)}+(C_{p}+R)^{2}\|\zeta\|_{L^{2}(D)}^{2}].
  \end{equation}
  Taking expectation on both sides, we have
  \begin{equation}
    \mathbb{E}\|u-u_{n}\|_{L^{2}(D)}^{2}\leq C[M \mathbb{E}\|\zeta\|_{L^{2}(D)}+(C_{p}+R)^{2} \mathbb{E}\|\zeta\|_{L^{2}(D)}^{2}],
  \end{equation}
where the constant $C$ depends only on $C_{p}$, $L$ and $M$, $R$.
  
The proof is completed.
\end{proof}
\end{theorem}
The above theorem shows that $u_{n}$ indeed approximates $u$, the solution of (\ref{Equ:umild}). 
\subsection{Error estimation for the multiscale method}
\label{The convergence analysis:un-ums}

\quad\quad The existence and uniqueness of the weak solution under the Assumption \ref{a1}-\ref{a3} had been proved in Ref. \cite{HC:MMS:2014}, i.e., there exists a unique weak solution $u(x, t) \in L^{2}(D)$ such that $\underset{t\in [0,T]}{\max} \mathbb{E} \|u(t)\|_{L^{2}(D)}^{2}<+\infty $, and for $v\in H^1_{0}(D)$, $t\in [0,T]$, it follows 
\begin{equation} \small
  (u(t),v)=  (u(0),v)-\int_0^t (\kappa \nabla u(s), \nabla v)ds +\int_0^t(f,v)ds+ \int_0^t (g(u(s)),v)ds+\int_0^t(\dot{W}(s),v)ds,
\end{equation}
where 
\begin{equation}
  \int_0^t(\dot{W}(s),v)ds:=\int_0^t (\sum_{k=1}^{\infty}\gamma_{k}\eta_{k}\chi_{k}(s),v)ds.
\end{equation}
\begin{definition}(Variational solution)
The variational solution $u(x,t)\in H^1_{0}(D)$ of the stochastic partial differential equation (\ref{u}) is defined by
 \begin{equation}\label{eq-uweak}
 \left\{
  \begin{aligned}
 (u_t,v) + a(u,v)& = (f,v)-(g(u),v)+(\dot{W},v
 ) \quad \forall v\in H^1_{0}(D), \, t> 0,\\
 \big(u(x,0), v\big)&=(u_{0},v) \quad \forall v\in H^1_{0}(D),
 \end{aligned}
    \right.
 \end{equation}
 where the bilinear form $a(u, v)$ is defined as follows
 \begin{equation}
 a(u,v) = \int_{D} \kappa \nabla u \nabla v\quad \forall~~u, v\in H^1_{0}(D).
 \end{equation}
 We use $(\cdot, \cdot)$ to denote the $L^{2}$ inner product, i.e.,
 \begin{equation}
 (f,v)=\int_{D}fv, \quad (g(u),v) = \int_{D} g(u)v, \quad (\dot{W},v)=\int_{D}\dot{W}v.
 \end{equation}
\end{definition}

Following the above procedure, the variational solution of problem (\ref{equ:un}) is to find $u_{n}(x,t) \in H^1_{0}(D)$ such that
  \begin{equation}\label{eq-uapproweak}
    \left\{
     \begin{aligned}
    ((u_{n})_t,v) + a(u_{n},v)& = (f,v)-(g(u_{n}),v)+(\dot{W}_{n},v
    ) \quad \forall v\in H^1_{0}(D), \, t> 0,\\
    \big(u_{n}(x,0), v\big)&=(u_{0},v) \quad \forall v\in H^1_{0}(D),
    \end{aligned}
       \right.
    \end{equation}
where $\dot{W}_{n}$ is the approximation of $\dot{W}$.

\textcolor{black}{We now consider a multiscale approximation of $u_{n}$. 
The multiscale solution of \ref{eq-uapproweak} is find $u_{n}^{ms}(x,t)\in V_{ms}$ such that}
\begin{equation}\label{eq-umsweak}
\left\{
 \begin{aligned}
\big((u_{n}^{ms})_t,v\big) + a(u_{n}^{ms},v)& = (f,v)-(g(u_{n}^{ms}),v)+(\dot{W}_{n},v) \quad \forall v\in V_{ms},\, t> 0,\\
\big(u_{n}^{ms}(x,0), v\big)&=(u_{0,ms},v) \quad \forall v\in V_{ms}.
\end{aligned}
   \right.
\end{equation}
where $u_{0,ms}\in V_{ms}$ is the $L^2$ projection of $u_0$ in $V_{ms}$.

To analyze the convergence of the proposed semi-implicit stochastic multiscale method for radiative heat transfer problem driven by the additive noises, we need  the following lemma.
     \begin{lemma}\label{lemw} 
      If the approximated white noise  $\dot{W}_{n}(t)=\sum_{k=1}^{\infty}\gamma_{k}^{n}\eta_{k}\psi_{k}(t)$, then
    \begin{equation}
      \mathbb{E}\int_0^T \| \dot{W}_{n}(t)\|_{L^{2}(D)}^{2}dt\leq C \sum_{k=1}^{\infty}(\gamma_{k}^{n})^{2},
    \end{equation}
    provided that the right-hand side is convergent.
    \end{lemma}
    
    The following lemma gives a desired bound for the solution of the parabolic equation (\ref{eq-uweak}).
    \begin{lemma}\label{lemu}
    Under Assumptions \ref{a1}-\ref{a3} and Lemma \ref{lemw}, the solution $u_{n}$ of the parabolic equation (\ref{equ:un}) satisfies the following estimates:\textcolor{black}{
      \begin{equation}
      \mathbb{E} \sup_{0\leq t\leq T}\|u_{n}\|_{L^{2}(D)}^2\leq C_{1} exp \int_0^T 2Ldt,
 \end{equation}
    }
    where $C_{1}=C(\|u_{0}\|_{L^{2}(D)}^2+\int_0^T \|f\|_{L^{2}(D)}^2 dt+\sum_{k=1}^{\infty}(\gamma_{k}^{n})^{2})$ is a constant independent of $\kappa(x)$ and coarse mesh size $H$.
    \begin{proof}
    In Eq.(\ref{eq-uapproweak}), taking $v=u_{n}$ and modulo appropriate stopping times:
    \begin{equation}
    ((u_{n})_{t},u_{n}) + a(u_{n},u_{n}) =(f,u_{n})-(g(u_{n}),u_{n})+(\dot{W}_{n},u_{n}),
    \end{equation}
    which implies
    \begin{equation}
    d\|u_{n}\|^2+ 2\|u_{n}\|_a^2 dt =2(f,u_{n})dt-2(g(u_{n}),u_{n})dt+2(\dot{W}_{n},u_{n})dt.
    \end{equation}
    Integrating with respect to time, we obtain
    \begin{equation}\label{Equ:un} 
    \begin{split}
    \|u_{n}\|_{L^{2}(D)}^2+2\int_0^T \|u_{n}\|_a^2 dt& =\|u_{0}\|_{L^{2}(D)}^2+2\int_0^T(f,u_{n})dt\\
    &-2\int_0^T(g(u_{n}),u_{n})dt+2\int_0^T(\dot{W}_{n},u_{n})dt.
    \end{split}
    \end{equation}
    Using the Assumptions \ref{a3}, we have
    \begin{equation}\label{Equ:gun}
    - \mathbb{E}\int_0^T (g(u_{n}),u_{n})dt \leq  L \mathbb{E} \int_0^T \|u_{n}\|_{L^{2}(D)}^2dt.
    \end{equation}
    Thanks to the Cauchy-Schwarz and Young's inequality, it follows that
    \begin{equation}\label{Equ:fun}
    \begin{split}
      \mathbb{E}\int_0^T (f,u_{n})dt& \leq  \mathbb{E}\int_0^T \|f\|_{L^{2}(D)}\|u_{n}\|_{L^{2}(D)}dt\\
    &\leq   \mathbb{E} \sup_{0\leq t\leq T}\|u_{n}(t)\|_{L^{2}(D)}\cdot\int_0^T \|f\|_{L^{2}(D)}dt\\
    &\leq \beta  \mathbb{E} \sup_{0\leq t\leq T}\|u_{n}(t)\|_{L^{2}(D)}^{2} +c(\beta )T\int_0^T \|f\|_{L^{2}(D)}^2dt.
    \end{split}
    \end{equation}
    Analogous to (\ref{Equ:fun}), we have
    \begin{equation}\label{Equ:Wn}
      \mathbb{E}\int_0^T (\dot{W}_{n},u_{n})dt \leq \delta  \mathbb{E} \sup_{0\leq t\leq T}\|u_{n}(t)\|_{L^{2}(D)}^{2} + c(\delta) \mathbb{E}\int_0^T\|\dot{W}_{n}\|_{L^{2}(D)}^2dt.
    \end{equation}
    Combining with Lemma \ref{lemw} and Eqs.(\ref{Equ:gun})-(\ref{Equ:Wn}), we get
    \begin{equation}
      \mathbb{E} \sup_{0\leq t\leq T}\|u_{n}\|_{L^{2}(D)}^2+2 \mathbb{E} \sup_{0\leq t\leq T}\int_0^T \|u_{n}\|_a^2dt\leq C_{1}+2L \mathbb{E}\int_0^T \|u_{n}\|_{L^{2}(D)}^2dt,
    \end{equation}
    where
    \begin{equation}
      C_{1}=C(\|u_{0}\|_{L^{2}(D)}^2+\int_0^T \|f\|_{L^{2}(D)}^2dt+\sum_{k=1}^{\infty}(\gamma_{k}^{n})^{2}).
    \end{equation}
   Noticed that
    \begin{equation}
     \mathbb{E} \sup_{0\leq t\leq T}\int_0^T \|u_{n}\|_a^2dt\geq 0.
    \end{equation}
    By the $Gronwall's$ inequality, we have\textcolor{black}{
   \begin{equation}
      \mathbb{E} \sup_{0\leq t\leq T}\|u_{n}\|_{L^{2}(D)}^2\leq C_{1} exp \int_0^T 2Ldt.
 \end{equation}
 }
    This completes the proof.
    \end{proof}
    \end{lemma}

    \begin{remark}\label{lemprojection} 
    Lemma \ref{lemu} implies that the function $g(u_{n})$ is also bounded according to the Assumptions \ref{a3}, i.e.,
      \begin{equation}
        \mathbb{E}\int_0^T \|g(u_{n})\|_{L^{2}(D)}^{2}dt \leq C*C_{1}exp \int_0^T 2Ldt.
      \end{equation}
    \end{remark}  

    \begin{lemma}\label{lemunt}
     Let $u_{n}$ be the solution of the parabolic equation (\ref{equ:un}). Under the conditions of Lemma \ref{lemw} and Lemma \ref{lemu}, we have
     \begin{equation}
      \mathbb{E} \int_0^T\|(u_{n})_{t}\|_{L^{2}(D)}^{2}dt\leq C_{2}
    \end{equation}
    for a constant $C_{2}=C(\|u_{0}\|_{L^{2}(D)}^2+\|u_{0}\|_a^{2}+\int_{0}^{T} \|f\|_{L^{2}(D)}^2dt+\sum_{k=1}^{\infty}(\gamma_{k}^{n})^{2})$ independent of $\kappa(x)$ and coarse mesh size $H$.
    \begin{proof}
   Taking $v=(u_{n})_{t}$ in Eq.(\ref{eq-uapproweak}), we get
    \begin{equation}
    ((u_{n})_{t},(u_{n})_{t}) + a(u_{n},(u_{n})_{t}) =(f,(u_{n})_{t})-(g(u_{n}),(u_{n})_{t})+(\dot{W}_{n},(u_{n})_{t}).
    \end{equation}
    Integrating with respect to time
    \begin{equation}\label{equntl2d} 
     \begin{split}
    \int_0^T \|(u_{n})_{t}\|_{L^{2}(D)}^2dt + \frac{1}{2} \|u_{n}\|_a^2 &= \frac{1}{2} \|u_0\|_a^2+ \int_0^T (f,(u_{n})_{t})dt\\
    &-\int_0^T(g(u_{n}),(u_{n})_{t})dt+\int_0^T (\dot{W}_{n},(u_{n})_{t})dt.
     \end{split}
     \end{equation}
    Thanks to the Assumptions \ref{a3} and the Cauchy-Schwarz and Young's inequality, we have
    \textcolor{black}{
    \begin{equation} 
       \begin{split}
      \mathbb{E} \int_0^T\|(u_{n})_{t}\|_{L^{2}(D)}^{2}dt+\frac{1}{2} \mathbb{E}\|u_{n}\|_a^2 &\leq C\bigg(\frac{1}{2}\|u_{0}\|_a^{2}+ \mathbb{E}\int_0^T\|f\|_{L^{2}(D)}^{2}dt\\
      &+ \mathbb{E}\int_0^T\|\dot{W}_{n}\|_{L^{2}(D)}^2dt+\mathbb{E}\sup_{0\leq t\leq T}\|u_{n}\|_{L^{2}(D)}^2\bigg).
     \end{split}
    \end{equation} } 
  Rely on Lemma \ref{lemw} and Lemma \ref{lemu}, we have
    \begin{equation}
      \mathbb{E} \int_0^T\|(u_{n})_{t}\|_{L^{2}(D)}^{2}dt+\frac{1}{2}\mathbb{E}\|u_{n}\|_a^2\leq C_{2},
    \end{equation}
    where
    \begin{equation}
      C_{2}=C(\|u_{0}\|_{L^{2}(D)}^2+\|u_{0}\|_a^{2}+\int_0^T \|f\|_{L^{2}(D)}^2dt+\sum_{k=1}^{\infty}(\gamma_{k}^{n})^{2}).
    \end{equation}
    {Since} $\mathbb{E}\|u_{n}\|_a^2\geq 0$, it holds that
     \begin{equation}
      \mathbb{E} \int_0^T\|(u_{n})_{t}\|_{L^{2}(D)}^{2}dt\leq C_{2}.
    \end{equation}
   This completes the proof.
    \end{proof}
    \end{lemma}

  Similar to the Lemma \ref{lemunt}, the CEM-GMsFEM solutions $u_{n}^{ms}$ satisfy the following Lemma.
  \begin{lemma}\label{lemunmst}
  Let $u_{n}^{ms}$ be the CEM-GMsFEM solution of the parabolic equation (\ref{eq-umsweak}). Under the conditions of Lemma \ref{lemw} and Lemma \ref{lemu}, we have
  \begin{equation}
    \mathbb{E} \int_0^T\|(u_{n}^{ms})_{t}\|_{L^{2}(D)}^{2}dt\leq C_{3}
   \end{equation}
   for a constant $C_{3}=C(\|u_{0,ms}\|_{L^{2}(D)}^2+\|u_{0,ms}\|_a^{2}+\int_0^T \|f\|_{L^{2}(D)}^2dt+\sum_{k=1}^{\infty}(\gamma_{k}^{n})^{2})$ independent of $\kappa(x)$ and coarse mesh size $H$.
  \end{lemma}

In order to estimate the error bound, we introduce the elliptic projection $\widehat{u}$, which bridges the approximation solution $u_{n}$ and the CEM-GMsFEM solutions $u_{n}^{ms}$.
\begin{lemma}
  \label{lem:u-uauxilary}
  Let $u_{n}$ be the solution of (\ref{equ:un}), and define its elliptic projection $\widehat{u}\in V_{ms}$ by 
  \begin{equation}\label{ms-ell}
    a\big(u_{n}-\widehat{u},v\big) =0, \quad \forall v\in V_{ms}.
  \end{equation}
  Then, by virtue of Lemma 1 in \cite{CEL:CMAME:2018}, we have
  \begin{equation}\label{ms-a-err}
    \mathbb{E}\int_0^T\| u_{n}-\widehat{u}\|_{a}^2dt\leq C_{4}H^{2}\Lambda^{-1}\kappa_{0}^{-1}.
   \end{equation}
   In particulal, we can derive that
   \begin{equation}\label{ms-l2-err}
    \mathbb{E}\int_0^T\| u_{n}-\widehat{u}\|_{L^{2}(D)}^{2}dt \leq C_{4}H^{4}\Lambda^{-2}\kappa_{0}^{-2},
  \end{equation}
where $\displaystyle \Lambda=\min_{1\leq i\leq N}\lambda_{L_{i}+1}^{(i)}$ and $\{ \lambda_{j}^{(i)}\} $ are the eigenvalues obtained from the construction of multiscale basis functions (\ref{cem-auxproblem}). Here $C_{4}=C(\|u_{0}\|_{L^{2}(D)}^2+\|u_{0}\|_a^{2}+\int_{0}^{T}\|f\|_{L^{2}(D)}^{2}dt+\sum_{k=1}^{\infty}(\sigma_{k}^{n})^{2})$ is a constant independent of $\kappa$ and coarse mesh size $H$.
\end{lemma}
  \begin{proof}
  First, we derive the error estimate for $u_{n}-\widehat{u}$ in the energy norm. Note that the solution $u_{n}\in H^1_{0}(D)$ of Eq.(\ref{eq-uapproweak}) satisfies
  \begin{equation}
  a(u_{n},v) = (f-(u_{n})_{t}-g(u_{n})+\dot{W}_{n},v), \quad \forall v\in H^1_{0}(D).
  \end{equation}
  Thus, $\widehat{u}\in V_{ms}$ satisfies
  \begin{equation}
  a(\widehat{u},v) = (f-(u_{n})_{t}-g(u_{n})+\dot{W}_{n},v), \quad \forall v\in V_{ms}.
  \end{equation}
  By means of Lemma 1 in \cite{CEL:CMAME:2018} and the inequality technique, we have 
  \begin{equation}
  \| u_{n}-\widehat{u}\|_{a}^2\leq CH^{2}\Lambda^{-1}\kappa_{0}^{-1}\big(\|f\|_{L^{2}(D)}^{2}+\|(u_{n})_{t}\|_{L^{2}(D)}^{2}+\|g(u_{n})\|_{L^{2}(D)}^{2}+\|\dot{W}_{n}\|_{L^{2}(D)}^{2}\big).
  \end{equation}
Integrating with respect to time, taking expectation and using Lemma \ref{lemu} and Lemma \ref{lemunt}, we have 
  \begin{equation}\label{equ:lemma8}
  \begin{split}
    \mathbb{E} \int_0^T\|u_{n}-\widehat{u}\|_{a}^2dt &\leq C_{4}H^{2}\Lambda^{-1}\kappa_{0}^{-1},\quad\forall t>0.
  \end{split}
  \end{equation}
 For presentation convenience, we define
  $$C_{4}=C(\|u_{0}\|_{L^{2}(D)}^2+\|u_{0}\|_a^{2}+\int_0^T\|f\|_{L^{2}(D)}^{2}dt+\sum_{k=1}^{\infty}(\sigma_{k}^{n})^{2}).$$\\
  Next, we derive the error estimate for $u_{n}-\widehat{u}$ in the sense of $L^{2}$-norm. We will apply the \text{Aubin-Nitsche} lift technique. For each $t>0$, we define $w\in H^1_{0}(D)$ by
  \begin{equation}
  a(w,v) = (u_{n}-\widehat{u}, v), \quad \forall v\in H^1_{0}(D),
  \end{equation}
  and define $\widehat{w}$ as the elliptic projection of $w$ in the space $V_{ms}$, that is,
  \begin{equation}
  a(\widehat{w},v) = (u_{n}-\widehat{u}, v), \quad \forall v\in V_{ms}.
  \end{equation}
  Taking $v=u_{n}-\widehat{u}\in V_{ms}$, we can deduce that
  \begin{equation}
  \begin{split}
  \| u_{n}-\widehat{u}\|_{L^{2}(D)}^2 & = a(w,u_{n}-\widehat{u})\\
  &= a(w-\widehat{w},u_{n}-\widehat{u})\\
  &\leq \|w-\widehat{w}\|_{a} \, \|u_{n}-\widehat{u}\|_{a}\\
  &\leq (CH\Lambda^{-\frac{1}{2}} \kappa_{0}^{-\frac{1}{2}}\|u_{n}-\widehat{u}\|_{L^{2}(D)}) \|u_{n}-\widehat{u}\|_{a},
  \end{split}
  \end{equation}
  i.e.,
  \begin{equation}\label{equ:un-uaux}
  \| u_{n}-\widehat{u}\|_{L^{2}(D)} \leq CH\Lambda^{-\frac{1}{2}} \kappa_{0}^{-\frac{1}{2}}\|u_{n}-\widehat{u}\|_{a}.
  \end{equation}
  Integrating with respect to time and taking expectation on both sides of Eq.(\ref{equ:un-uaux}), we get
  \begin{equation}
    \mathbb{E}\int_0^T\| u_{n}-\widehat{u}\|_{L^{2}(D)}^{2}dt \leq CH^{2}\Lambda^{-1}\kappa_{0}^{-1}\mathbb{E}\int_0^T\|u_{n}-\widehat{u}\|_{a}^{2}dt.
  \end{equation}
  Combing with Eq.(\ref{equ:lemma8}), we have
  \begin{equation}
    \mathbb{E}\int_0^T\| u_{n}-\widehat{u}\|_{L^{2}(D)}^{2}dt \leq C_{4}H^{4}\Lambda^{-2}\kappa_{0}^{-2}, \quad\forall t>0.
  \end{equation}
 This completes the proof.
  \end{proof}
Then, we give an error estimate between the approximating solution $u_{n}$ and CEM-GMsFEM soluton $u_{n}^{ms}$ in the sense of $L^{2}$-norm.

\begin{theorem}\label{unums}
Let $u_{n}$ and $u_{n}^{ms}$ be the solution of \ref{eq-uapproweak} and {eq-umsweak}, respectively; then the error estimate is
 \begin{equation}
 \mathbb{E} \| (u_{n} - u_{n}^{ms})(x,T) \|_{L^{2}(D)}^2 \leq  C_{5}H^2\Lambda^{-1}\kappa_{0}^{-1}+\mathbb{E}\| (u_{n}-u_{n}^{ms})(x,0) \|_{L^{2}(D)}^2,
 \end{equation}
 where $\displaystyle \Lambda=\min_{1\leq i\leq N}\lambda_{L_{i}+1}^{(i)}$ and 
 \begin{equation}
 C_{5}=C\bigg(\|u_{0}\|_{L_{2}(D)}^{2}+\|u_{0,ms}\|_{L^{2}(D)}^{2}+\|u_{0}\|_a^{2}+\|u_{0,ms}\|_a^{2}+\int_0^T\|f\|_{L^{2}(D)}^{2}dt+ \sum_{k=1}^{\infty}(\gamma_{k}^{n})^{2}\bigg)
 \end{equation}
 is independent of $\kappa(x)$, $\sigma$ and coarse mesh size $H$.
\end{theorem}
 \begin{proof}: Following the same line of \cite{ZGJ:JCAM:2021}, we obtain
  \begin{equation}
  ((u_{n}-u_{n}^{ms})_t,v) + a(u_{n}-u_{n}^{ms},v) =-(g(u_{n})-g(u_{n}^{ms}),v)
  \end{equation}
  for all $v\in V_{ms}$.
  
  We define $\widehat{u}\in V_{ms}$ as the elliptic projection of $u$, which satisfies Eqs.(\ref{ms-a-err}) and (\ref{ms-l2-err}). Taking $v=\widehat{u}-u_{n}^{ms}=(u_{n}-u_{n}^{ms})-(u_{n}-\widehat{u})$, we have the following equation.
  \begin{equation}
  \begin{split}
  &\: ((u_{n}-u_{n}^{ms})_t,u_{n}-u_{n}^{ms})+a(u_{n}-u_{n}^{ms},u_{n}-u_{n}^{ms})\\
  &\:=((u_{n}-u_{n}^{ms})_t,u_{n}-\widehat{u})+a(u_{n}-u_{n}^{ms},u_{n}-\widehat{u})\\
  &\: -(g(u_{n})-g(u_{n}^{ms}),u_{n}-u_{n}^{ms})+(g(u_{n})-g(u_{n}^{ms}),u_{n}-\widehat{u}).
  \end{split}
  \end{equation}
  This yields
  \begin{equation}\label{equ:abcd}
  \begin{split}
  &\: \frac{1}{2} \frac{d}{dt} \|u_{n}- u_{n}^{ms} \|_{L^{2}(D)}^2 +\|u_{n}- u_{n}^{ms} \|_{a}^2\\
  &\:=((u_{n}-u_{n}^{ms})_t,u_{n}-\widehat{u})+a(u_{n}-u_{n}^{ms},u_{n}-\widehat{u})\\
  &\: -(g(u_{n})-g(u_{n}^{ms}),u_{n}-u_{n}^{ms})+(g(u_{n})-g(u_{n}^{ms}),u_{n}-\widehat{u}).
  \end{split}
  \end{equation}
  Applying Assumption \ref{a3}, the Cauchy-Schwarz and Young's inequality, we have
  \begin{equation}\label{equ:a}
  \begin{split}
  ((u_{n}-u_{n}^{ms})_t,u_{n}-\widehat{u})&\leq \|(u_{n})_t-(u_{n}^{ms})_t\|_{L^{2}(D)}\|u_{n}-\widehat{u}\|_{L^{2}(D)}\\
  &\leq(\|(u_{n})_t\|_{L^{2}(D)}+\|(u_{n}^{ms})_t\|_{L^{2}(D)})\|u_{n}-\widehat{u}\|_{L^{2}(D)},
  \end{split}
  \end{equation}
  \begin{equation}\label{equ:b}
  \begin{split}
  a(u_{n}-u_{n}^{ms},u_{n}-\widehat{u})&\leq \|u_{n}-u_{n}^{ms}\|_{a}\|u_{n}-\widehat{u}\|_{a}\\
  &\leq\frac{1}{2}\|u_{n}-u_{n}^{ms}\|_{a}^{2}+\frac{1}{2}\|u_{n}-\widehat{u}\|_{a}^{2},
  \end{split}
  \end{equation}
  \begin{equation}\label{equ:c}
  \begin{split}
  -(g(u_{n})-g(u_{n}^{ms}),u_{n}-u_{n}^{ms})&=\int_{D}(g(u_{n})-g(u_{n}^{ms}))(u_{n}-u_{n}^{ms})dx \\ &\leq L\|u_{n}- u_{n}^{ms}\|_{L^{2}(D)}^2,
  \end{split}
  \end{equation}
  \begin{equation}\label{equ:d}
  \begin{split}
  (g(u_{n})-g(u_{n}^{ms}),u_{n}-\widehat{u})&=\int_{D}(g(u_{n})-g(u_{n}^{ms}))(u_{n}-\widehat{u})dx \\
  &\leq M\|u_{n}-\widehat{u}\|_{L^{2}(D)}+\frac{R}{2}(\|u_{n}- u_{n}^{ms} \|_{L^{2}(D)}^2+\|u_{n}-\widehat{u}\|_{L^{2}(D)}^{2}).
  \end{split}
  \end{equation}
 If we use Eqs.(\ref{equ:a}), (\ref{equ:b}),(\ref{equ:c}) and (\ref{equ:d}) in Eq.(\ref{equ:abcd}), then we have 
  \begin{equation}
  \begin{split}
  &\:  \frac{d}{dt} \|u_{n}- u_{n}^{ms}\|_{L^{2}(D)}^2 +\|u_{n}-u_{n}^{ms}\|_{a}^2\\
  &\:\leq(2L+R)(\|u_{n}-u_{n}^{ms} \|_{L^{2}(D)}^2+R\|u_{n}-\widehat{u}\|_{L^{2}(D)}^{2}\\
  &\:+\|u_{n}-\widehat{u}\|_{a}^{2}+2(\|(u_{n)})_t\|_{L^{2}(D)}+\|(u_{n}^{ms})_t\|_{L^{2}(D)}+M)\|u_{n}-\widehat{u}\|_{L^{2}(D)}.
  \end{split}
  \end{equation}
  Combining Lemma \ref{lemunmst} and Lemma \ref{lem:u-uauxilary}, integrating with respect to time and taking expectation of both sides, we have
  \begin{equation}
  \begin{split}
  &\: \mathbb{E} \| (u_{n}- u_{n}^{ms})(x,T) \|_{L^{2}(D)}^2 + \mathbb{E}\int_0^T \|u_{n}- u_{n}^{ms}\|_{a}^2dt\\
  &\: \leq \mathbb{E}\int_0^T(2L+R)\| u_{n} -u_{n}^{ms}\ \|_{L^{2}(D)}^2dt+\mathbb{E}\| (u_{n}- u_{n}^{ms})(x,0) \|_{L^{2}(D)}^2 \\&\:+C_{5}H^2\Lambda^{-1}\kappa_{0}^{-1}(1+CH^2\Lambda^{-1}\kappa_{0}^{-1}).
  \end{split}
  \end{equation}
 For presentation convenience, we define
  \begin{equation}
    C_{5}=C\bigg(\|u_{0}\|_{L_{2}(D)}^{2}+\|u_{0,ms}\|_{L^{2}(D)}^{2}+\|u_{0}\|_a^{2}+\|u_{0,ms}\|_a^{2}+\int_0^T\|f\|_{L^{2}(D)}^{2}dt+\sum_{k=1}^{\infty}(\gamma_{k}^{n})^{2}\bigg).
  \end{equation}
  Taking into account that $\mathbb{E}\int_0^T \|u_{n}- u_{n}^{ms}\|_{a}^2\geq 0$ and using the $Gronwall's$ inequality, we have
  \begin{equation}
    \mathbb{E}\| (u_{n} -u_{n}^{ms})(x,T) \|_{L^{2}(D)}^2\leq  C_{5}H^2\Lambda^{-1}\kappa_{0}^{-1}+\mathbb{E}\| (u_{n}- u_{n}^{ms})(x,0) \|_{L^{2}(D)}^2.
  \end{equation}
  This completes the proof.
  \end{proof}
  
Furthermore, combining Theorem \ref{thmapproerror} and Theorem \ref{unums}, an estimate on 
$\mathbb{E}\|(u - u_{n}^{ms})(x,T)\|_{L^{2}(D)}^{2}$
follows from the triangle inequality.
\begin{theorem}\label{them3}
\label{thmfinal}
\textcolor{black}{Let $u$ and $u_{n}^{ms}$ be the solution of (\ref{u}) and (\ref{eq-umsweak}), respectively. Then we can prove that}
\begin{equation}
 \begin{split}
  \mathbb{E}\|(u - u_{n}^{ms})(x,T)\|_{L^{2}(D)}^{2}& \leq C\bigg(M\|\zeta\|_{L^{2}(D)}+(C_{p}+R)\|\zeta\|_{L^{2}(D)}^{2}\bigg)\\
  & +C_{5}H^2\Lambda^{-1}\kappa_{0}^{-1}+\mathbb{E}\| (u - u_{n}^{ms})(x,0) \|_{L^{2}(D)}^2,
\end{split}
\end{equation}
\textcolor{black}{where $C_{5}$ is defined as before. The constants $M$, $C_{p}$, and $R$ are from Assumption \ref{a3}}.
\end{theorem}

\textcolor{black}{The error estimation of this numerical scheme consists of three parts. The first part originates from the truncation of the infinite-dimensional noise $W(t)$, which can decrease by choosing more truncated terms. The second part originates from the approximated error of the spatial multiscale space $V_{ms}$, which can decrease by choosing more degrees of freedom in each local domain. The third part originates from the projection error of the initial value $u_0(x)$ in multiscale space $V_{ms}$. 
}
\section{Numerical results}

\qquad \textcolor{black}{In this section, we present a few representative numerical examples to evaluate the performance of the proposed semi-implicit stochastic multiscale method for solving the stochastic radiative heat transfer equation. We mainly focus on the verification of the accuracy and efficiency of the CEM-GMsFEM of the semi-implicit predictor-corrected numerical scheme.}
We will use the backward Euler scheme for time discretization. In subsection \ref{Num-examp1}, we present some numerical examples for Eq.(\ref{u}) in a bounded domain with a periodic microstructure of period $\varepsilon=\frac{1}{8}$ to demonstrate the convergence of CEM-GMsFEM. In subsection \ref{Num-examp2}, the effects of the problem (\ref{u}) in a nonperiodic domain with  different rapidly oscillating coefficients are considered.

 We consider the problems (\ref{u}) on the domain $D=[0,1]^2$, and the final computational time is  $T=0.1$. The source term $f(x,t)$ and initial condition $u_{0}(x)$ are defined by
\begin{equation}
  f(x,t)=3\pi^{2}e^{\pi^{2}t}\sin(\pi x_{1})\sin(\pi x_{2}),
\end{equation}
\begin{equation}
 u_{0}(x)=\sin(\pi x_{1})\sin(\pi x_{2}).
\end{equation}

For stochastic radiative heat transfer Eq.(\ref{u}), we choose the noise in the form of (\ref{noise}) with the following coefficient $\gamma_{k}$ (Ref.\cite{DZ:SIAMJNA:2002}):
\begin{equation}
\gamma_{k}=\frac{1}{k^{3/2}},\quad
\gamma_{k}^{n}=\left\{
\begin{aligned}
\gamma_{k} & , & if\quad k\leq n, \\
~0 & , & if\quad k> n.
\end{aligned}
\right.
\end{equation}
Thus, by the above and Lemma \ref{thmzeta}, we have 
\begin{equation}\label{zeta11}
  \mathbb{E} \|\zeta(t)\|_{L^{2}(D)}^{2} \leq C \|\bm{\gamma}-\bm{\gamma^{n}}\|_{Q_{-1}}^{2}=C\sum_{k=n+1}^{\infty} (\frac{1}{k^{3/2}}\cdot \frac{1}{k})^{2}\leq C\frac{1}{n^{4}},\quad \forall t> 0.
\end{equation}
In the approximate representations of white noises, we set the deterministic functions $\chi_{k}(t)=\sqrt{2}\sin(k\pi t)$ in Eq.(\ref{appnoise}) , which form an orthonormal basis in $L_{2}([0,T])$. Then Theorem \ref{them3} can be rewritten as follows
\begin{equation}\label{CEMconvergence}
  \begin{split}
   \mathbb{E}\|(u - u_{n}^{ms})(x,T)\|_{L^{2}(D)}^{2}& \leq C\bigg(M\frac{1}{n^{2}}+(C_{p}+R)\frac{1}{n^{4}}\bigg)\\
   & +C_{5}H^2\Lambda^{-1}\kappa_{0}^{-1}+\mathbb{E}\| (u - u_{n}^{ms})(x,0) \|_{L^{2}(D)}^2.
  \end{split}
\end{equation}
To quantify the accuracy of the CEM-GMsFEM, we define relative $L_{2}$ norm error and weighted energy error at the time $T$:
\begin{equation}
  \varepsilon_{L_{2}}=\frac{\|\mathbb{E}(u_{n}^{ms}(x,T))-\mathbb{E}(u_{h}(x,T))\|_{L_{2}(D)}}{\|\mathbb{E}(u_{h}(x,T))\|_{L_{2}(D)}},
\end{equation} 
and
\begin{equation}
  \varepsilon_{a}=\frac{\|\mathbb{E}(u_{n}^{ms}(x,T))-\mathbb{E}(u_{h}(x,T))\|_{a}}{\|\mathbb{E}(u_{h}(x,T))\|_{a}},
\end{equation} 
where $\mathbb{E}(u_{h}(x,T))$ and $\mathbb{E}(u_{n}^{ms}(x,T))$ are the expectations of the reference and CEM-GMsFEM solution at the final moment $T$ over all discretized Brownian paths, and the number of Brownian paths is 100. For the spatial discretization, $100\times 100$ uniform fine grids in domain $D$ are fixed to compute the reference solution, and different coarse grids size are used for CEM-GMsFEM solution. For numerical computation of additive white noises, the random number generator \textbf{randn} is used to generate  independent "pseudorandom" numbers from standard normal  distribution. To compute the mean, we use Monte Carlo sampling for the numerical examples  with the Mersenne Twister random generator(seed 100).
\subsection{Example 1}\label{Num-examp1}
\qquad 
In this example, we show the numerical results of solving Eq. (\ref{u}) with CEM-GMsFEM in a bounded domain $D$ with a periodic microstructure with period $\varepsilon=\frac{1}{8}$. The heat conduction coefficient $\kappa(x)=\kappa_{1}(x)$  is depicted in the left side of Fig. \ref{example1-coef-basis}. The Stefan-Boltzmann coefficients are chosen to have the same value as the heat conduction coefficient. The number of oversampling layers is $4*(\log(H)/\log(1/10))$, and we take $6$ basis functions on each coarse block (see Ref.\cite{ZGJ:JCAM:2021}). To investigate the influence of using different numbers of basis functions, we use a fixed coarse grid with grid size $H =1/10$ and a fixed number of oversampling layers $m=4$. We observe that using more local basis functions will definitely lead to a more accurate coarse solution due to the high contrast, which agrees with the observations from the right side of Fig. \ref{example1-coef-basis}. Once the number of local basis functions exceeds a certain number, the errors decay slower. This happens when the decay of the eigenvalues slows down.
\begin{figure}[ht]
  \centering
  \includegraphics[width=2.2in, height=2in]{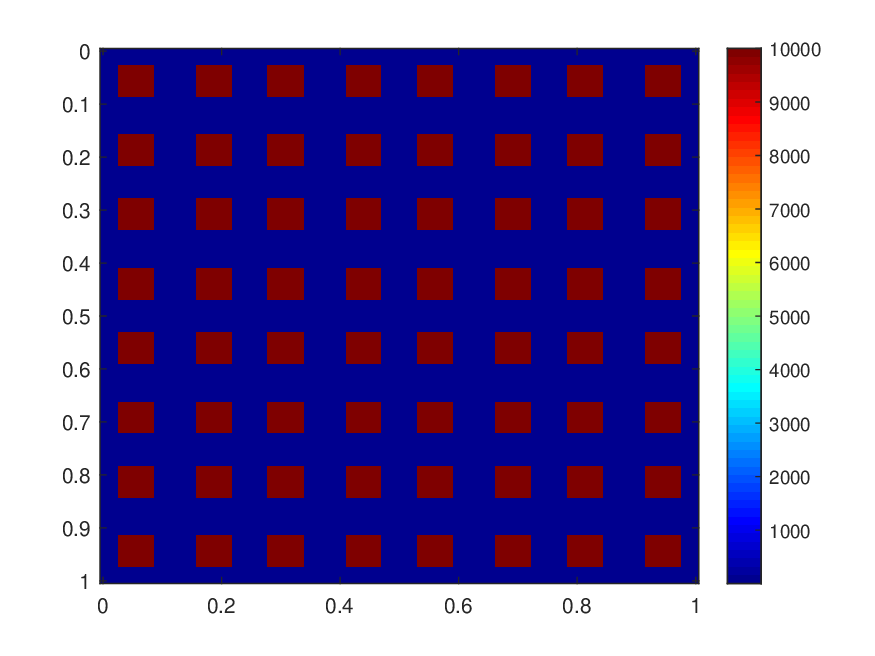}
  \includegraphics[width=2.2in, height=2in]{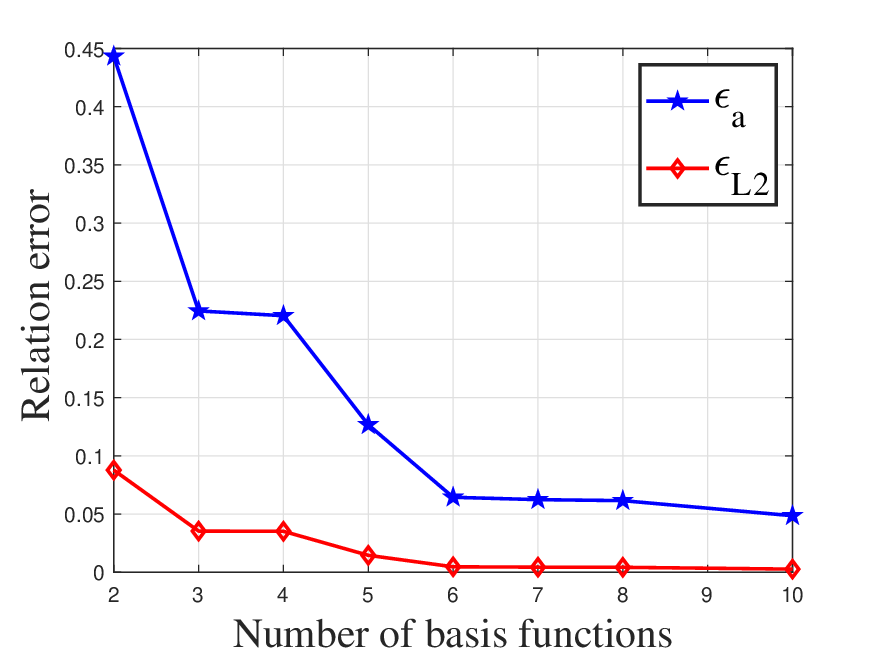}
  \caption{Map of field $\kappa(x)$(left), numerical results with various numbers of basis functions(right).}
  \label{example1-coef-basis}
 \end{figure}
\begin{figure}[htpb]
  \subfigure[ ]{
    \begin{minipage}[t]{0.3\linewidth}
    \centering
    \includegraphics[width=2in, height=1.9in]{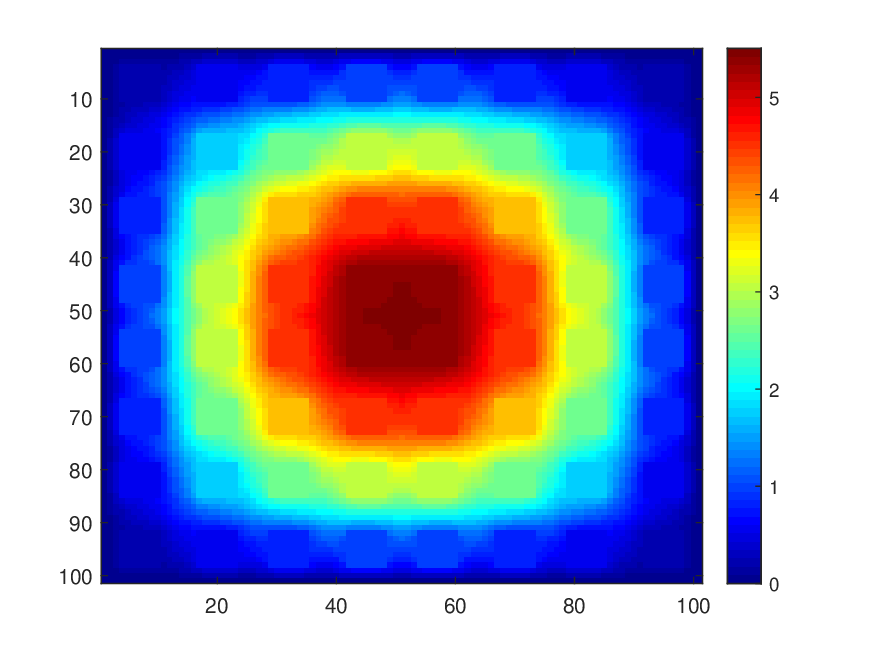}
    \end{minipage}
  }
  \subfigure[ ]{
    \begin{minipage}[t]{0.3\linewidth}
    \centering
    \includegraphics[width=2in, height=1.9in]{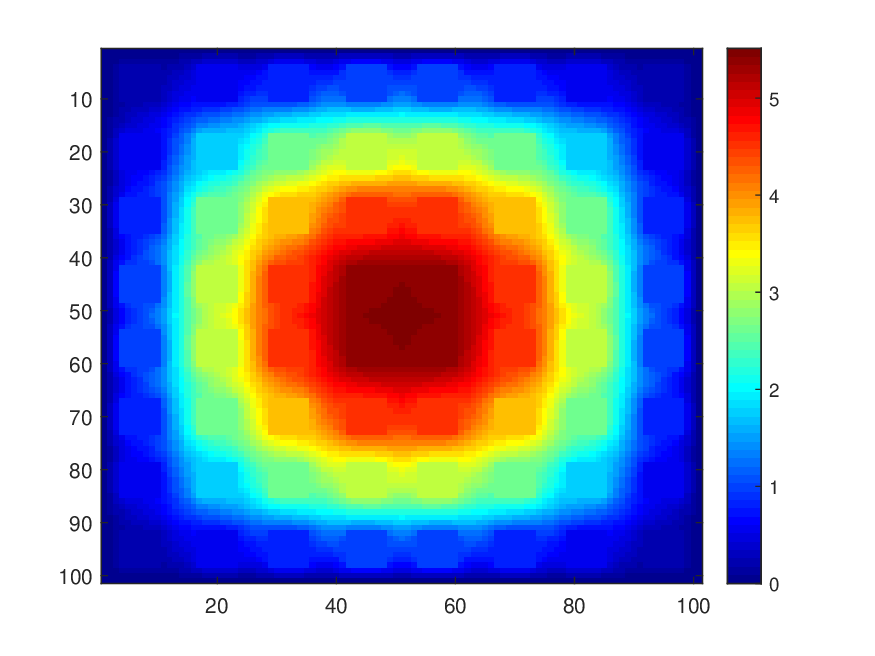}
    \end{minipage}
  }
  \subfigure[ ]{
     \begin{minipage}[t]{0.3\linewidth}
    \centering
    \includegraphics[width=2in, height=1.9in]{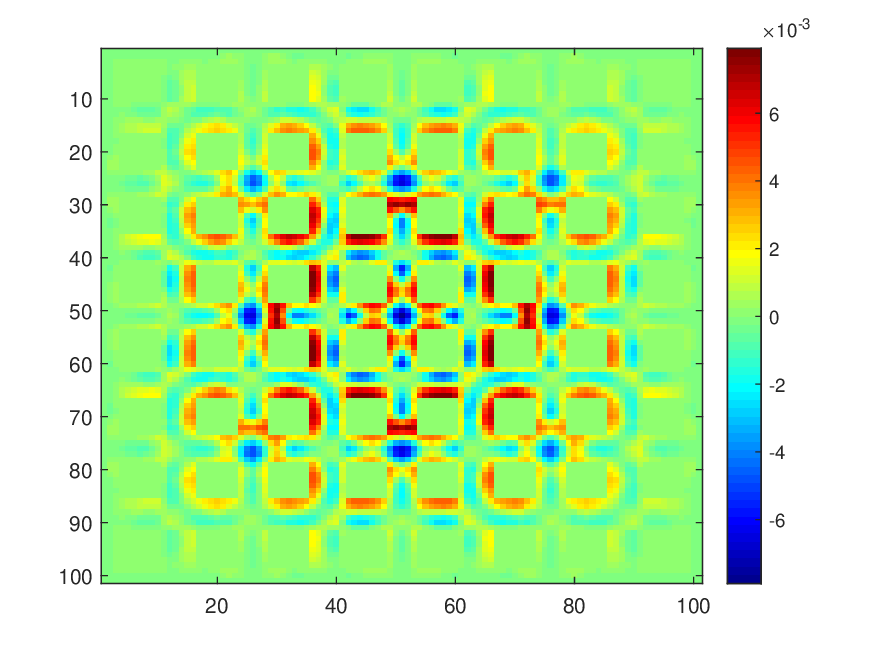}
    \end{minipage}
  }
    \caption{(a) the mean of reference solution at $T=0.1$ (b) the mean of CEM-GMsFEM solution at $T=0.1$ (c) the pointwise error between (a) and (b).}
    \label{example1-solutions}
  \end{figure}

Figs. \ref{example1-solutions}(a) and (b) depict the numerical mean solutions by the fine-scale formulation and the coarse-scale formulation at the final time $T = 0.1$, respectively. The corresponding pointwise error $\mathbb{E}(u_{n}^{ms}(x,T))-\mathbb{E}(u_{h}(x,T))$ at $T=0.1$ are plotted in Fig. \ref{example1-solutions}(c) to verify the accuracy of our proposed methods. One can see that the proposed method captures most of the details of the reference solution and provides very good accuracy at a reduced computational expense.
\begin{table}[htpb]
  \centering
  \caption{Comparison ($\varepsilon_{a}$) of various number of oversampling layers and different contrast values, $H=1/10$, $a=6$.}
  \vspace{1pt}
  \begin{tabular}{|c|c|c|c|c|c|c|}
    \hline
    \diagbox{layers}{contrast}   & $10^4$               & $10^5$            & $10^6$            & $10^7$             & $10^8$       & $10^9$ \\
    \hline
     $3$                        & $6.4408e$-$02$     & $1.3561e$-$02$   &  $1.3562e$-$02$    &  $1.3554e$-$02$    & $1.3559e$-$02$  &  $1.3560e$-$02$\\
    \hline
     $4$                        & $6.4409e$-$02$      & $1.2244e$-$02$  &  $1.2247e$-$02$    &  $1.2243e$-$02$   &  $1.2243e$-$02$   &  $1.2243e$-$02$\\
    \hline
     $5$                        & $6.4382e$-$02$      & $1.2201e$-$02$  &  $1.2204e$-$02$    &  $1.2200e$-$02$   &  $1.2200e$-$02$   &  $1.2199e$-$02$\\
    \hline
  \end{tabular}
  \label{example1-contrastness}
  \end{table}

In Tab. \ref{example1-contrastness}, the performance of the proposed methods with respect to the relation between contrast values and the number of oversampling layers is presented. We compare the energy error ($\varepsilon_{a}$) with various contrast values from $10^{4}$ to $10^{9}$, where the conductivity values in the channels are the same, with $6$ basis functions per coarse region and coarse mesh size of $H = 1/10$.
\begin{figure}[ht]
  \centering
  \includegraphics[width=2.2in, height=2in]{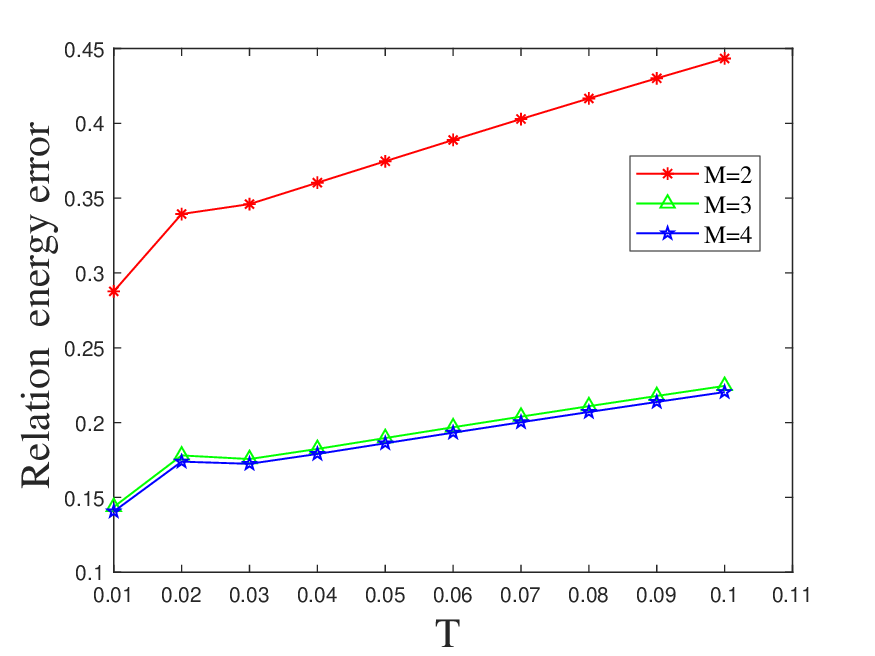}
  \includegraphics[width=2.2in, height=2in]{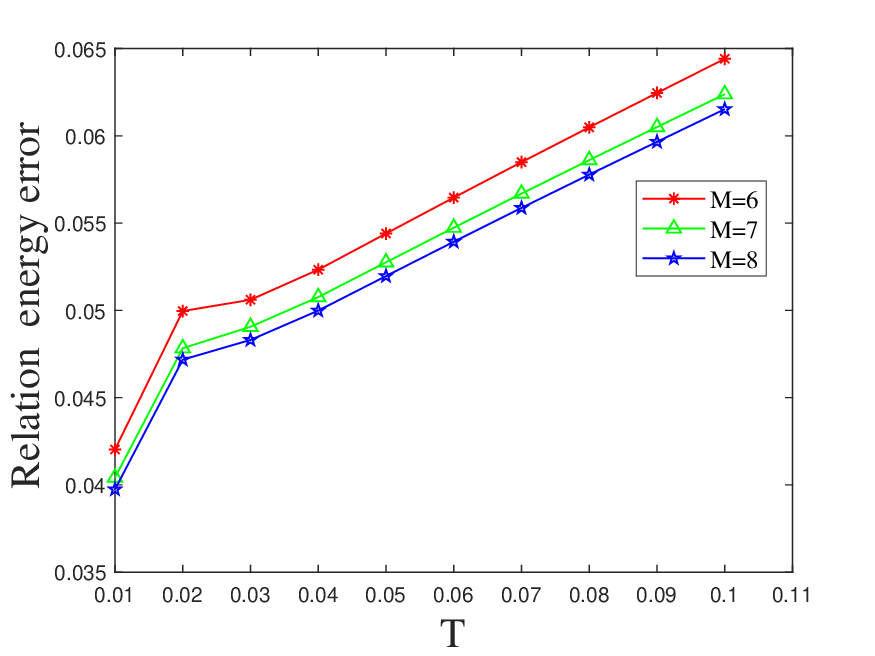}
  \caption{Relative energy error under different time $T$ with various numbers of 
  basis functions in each coarse region $M=2,3,4,6,7,8$, where time step $\Delta t$ is $0.001$.}
  \label{dtenegyerror}
 \end{figure}

Fig. \ref{dtenegyerror} presents the properties of the relative energy error for our proposed method. We plot the relative energy error at time $T=[0,0.1]$ under various numbers of basis functions in each coarse region $M=2,3,4,6,7,8$, where time step $\Delta t$ is $0.001$. We can see the improvements in accuracy by enriching the basis functions in each coarse region from $M=2$ to $M=4$ and from $M=6$ to $M=8$. As time $T$ increases from $0.01$ to $0.1$, the energy errors at these times are adding up. The energy error has a cumulative effect as time $T$ increase, mainly because of the nonlinear part of the term. The nonlinear part, especially $\sigma(x)u^{4}$, makes any small error amplified.

\subsection{Example 2}\label{Num-examp2}
\qquad In this example, we use the coefficients depicted in Figs. \ref{exam2_no_coef} and \ref{exam2_over_coef} that corresponds to coefficients with background one and high conductivity inclusions within all circular regions with different centers and sizes. We consider two cases, corresponding to overlapping inclusions (Fig. \ref{exam2_over_coef}) and non-overlapping inclusions (Fig. \ref{exam2_no_coef}), respectively. The centers of these circular regions obey the uniform distribution in region $D = [0,1] \times [0,1]$, and the radii obey the uniform distribution in the interval [0.05, 0.1]. Unless otherwise specified, other parameters are the same as in the previous example.

\begin{figure}[htbp]
	\centering  
	\subfigure[N =10]{
		\label{exam2_no_coef1}
		\includegraphics[width=2in, height=1.9in]{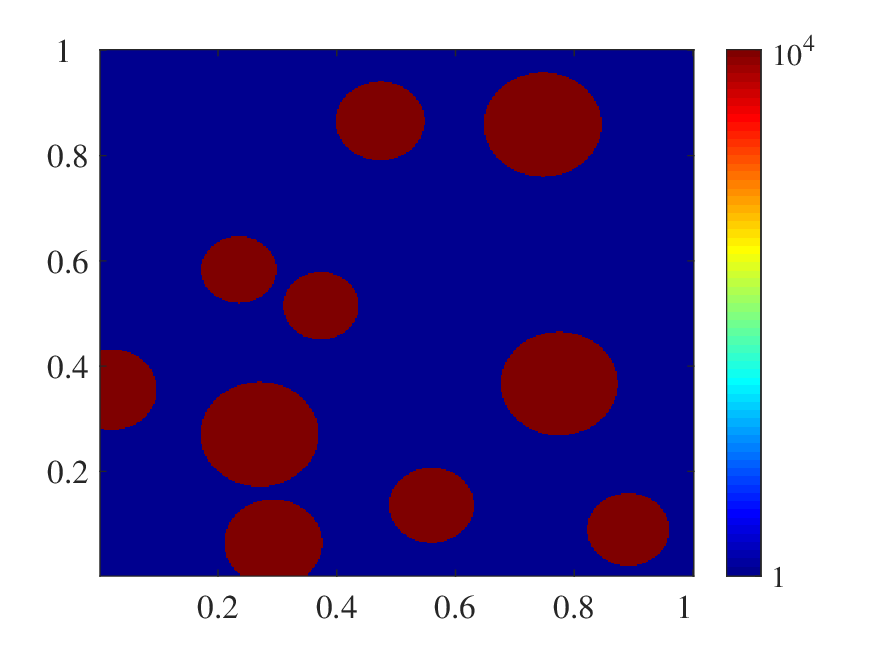}}
	\subfigure[N = 20]{
		\label{exam2_no_coef2}
		\includegraphics[width=2in, height=1.9in]{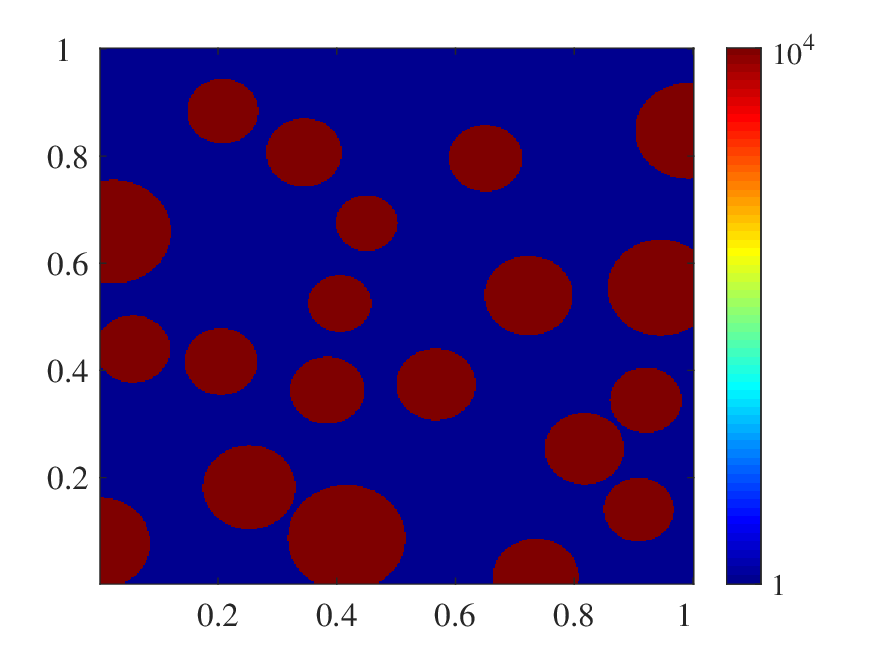}}
	\\
	\subfigure[N =30]{
		\label{exam2_no_coef3}
		\includegraphics[width=2in, height=1.9in]{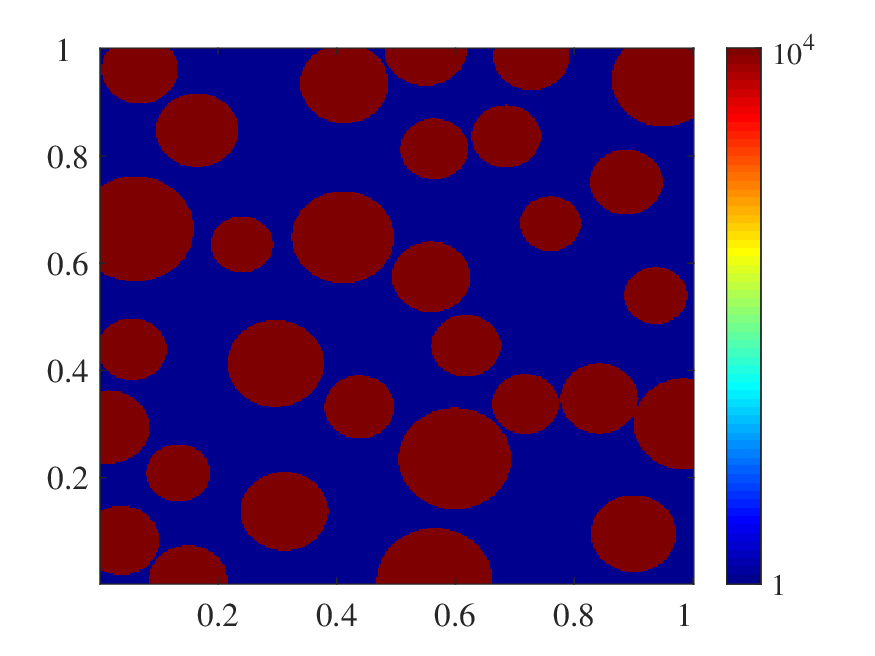}}
	\subfigure[N = 40]{
		\label{exam2_no_coef4}
		\includegraphics[width=2in, height=1.9in]{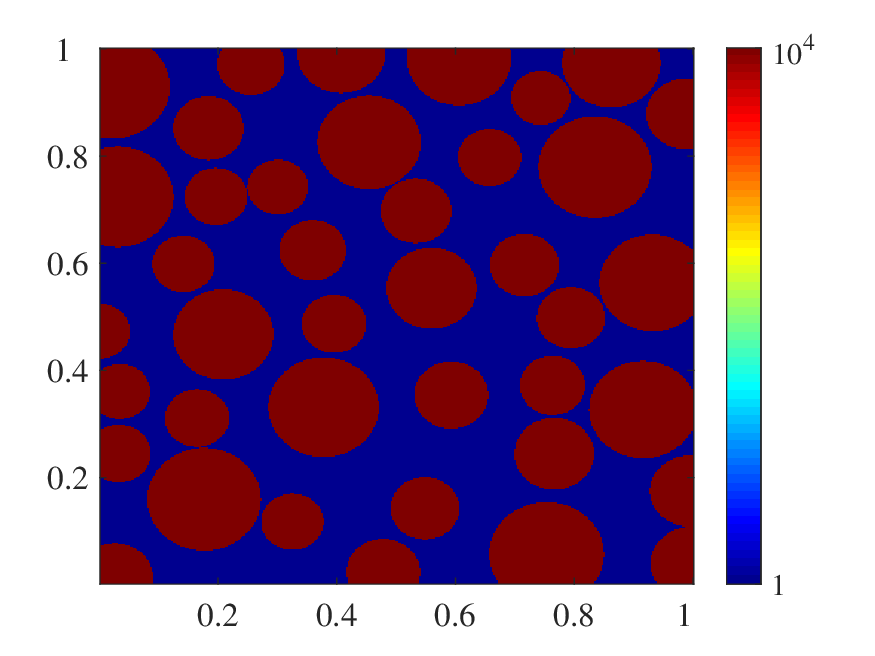}}
	\caption{Map of fields $\kappa(x)$ with various numbers of non-overlapping inclusions $N$.}
	\label{exam2_no_coef}
\end{figure}

\begin{table}[]
	\centering
	\caption{Comparison ($\varepsilon_{L_{2}}$ and $\varepsilon_{a}$) of various numbers of non-overlapping inclusions $N$.}
	\begin{tabular}{|cc|c|c|c|c|}
		\hline
		\multicolumn{2}{|c|}{N}                         & \multicolumn{1}{l|}{10} & \multicolumn{1}{l|}{20} & \multicolumn{1}{l|}{30} & \multicolumn{1}{l|}{40} \\ \hline
		\multicolumn{1}{|c|}{\multirow{2}{*}{a=3}} & $\varepsilon_{L_{2}}$ & 1.9819e-3               & 3.3391e-3               & 4.8637e-3               & 1.2888e-2               \\ \cline{2-6} 
		\multicolumn{1}{|c|}{}                     & $\varepsilon_{a}$ & 2.4825e-2               & 3.6259e-2               & 4.1764e-2               & 7.2074e-2               \\ \hline
		\multicolumn{1}{|c|}{\multirow{2}{*}{a=6}} & $\varepsilon_{L_{2}}$ & 5.9002e-4               & 9.1005e-4               & 1.6291e-3               & 3.0949e-3               \\ \cline{2-6} 
		\multicolumn{1}{|c|}{}                     & $\varepsilon_{a}$ & 1.0466e-2               & 1.4410e-2               & 2.0573e-2               & 3.0284e-2               \\ \hline
	\end{tabular}
	\label{exam2_no_numbers}
\end{table}
\begin{figure}[htpb]
	\subfigure[ ]{
		\begin{minipage}[t]{0.3\linewidth}
			\centering
			\includegraphics[width=2in, height=1.9in]{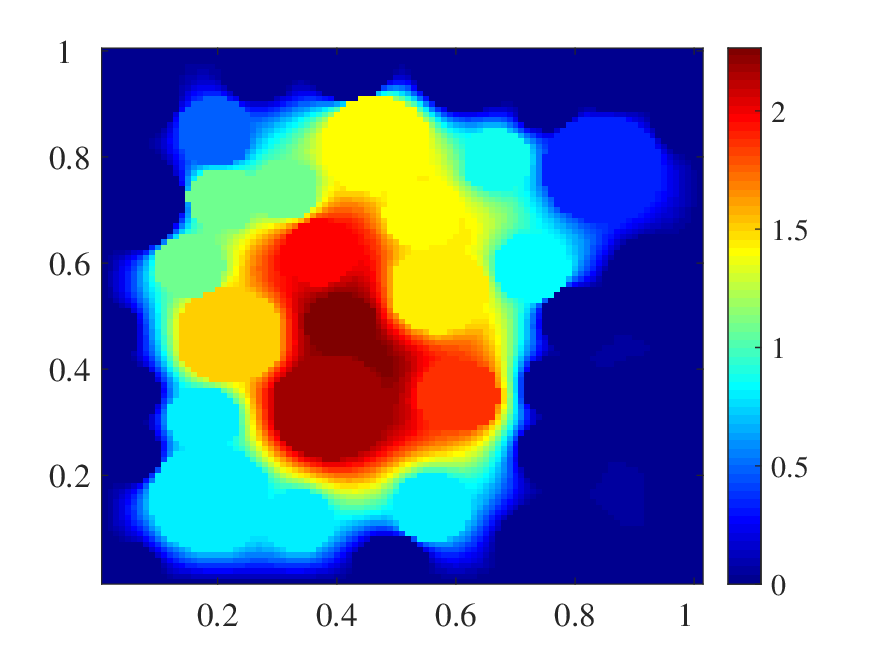}
		\end{minipage}
	}
	\subfigure[ ]{
		\begin{minipage}[t]{0.3\linewidth}
			\centering
			\includegraphics[width=2in, height=1.9in]{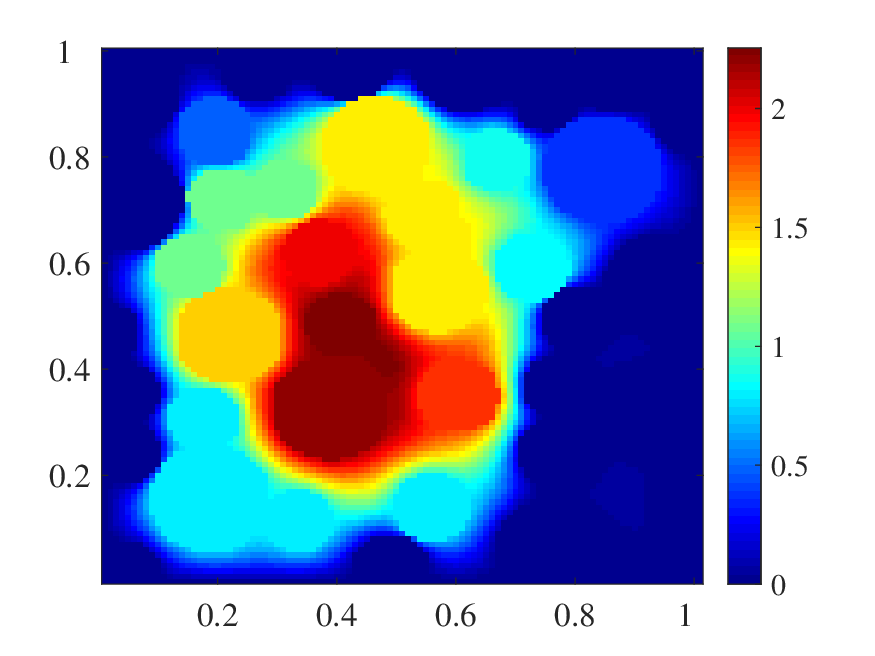}
		\end{minipage}
	}
	\subfigure[ ]{
		\begin{minipage}[t]{0.3\linewidth}
			\centering
			\includegraphics[width=2in, height=1.9in]{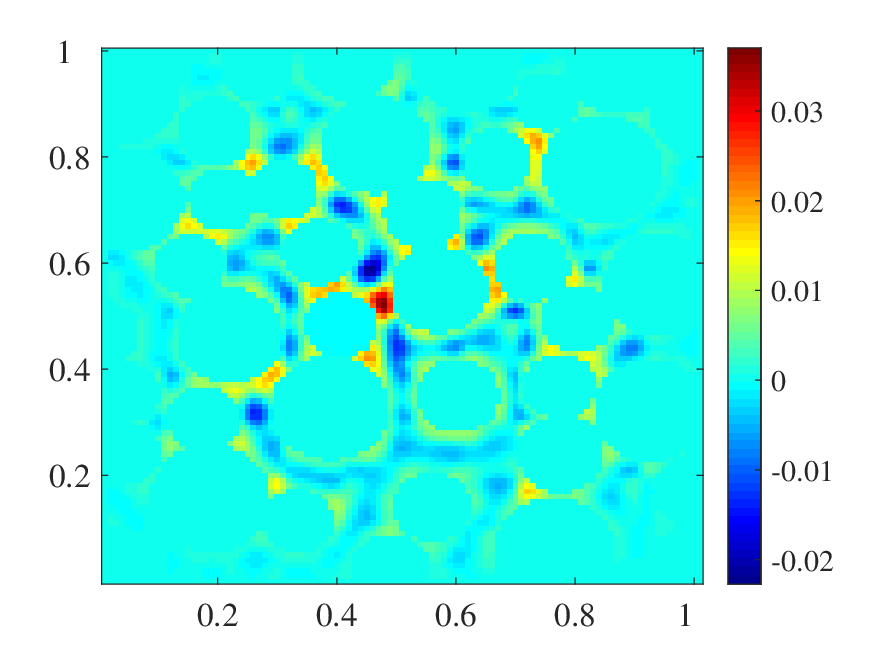}
		\end{minipage}
	}
	\caption{(a) the mean of reference solution at $T=0.1$ (b) the mean of CEM-GMsFEM solution at $T=0.1$ (c) the pointwise error between (a) and (b), where $\kappa(x)$ is taken from Fig. \ref{exam2_no_coef}(d) and N = 40, M = 6.}
	\label{example2_no_solutions}
\end{figure}

\begin{figure}[htbp]
	\centering  
	\subfigure[N =10]{
		\label{exam2_over_coef1}
    \begin{minipage}[t]{0.3\linewidth}
      \centering
		\includegraphics[width=2in, height=1.9in]{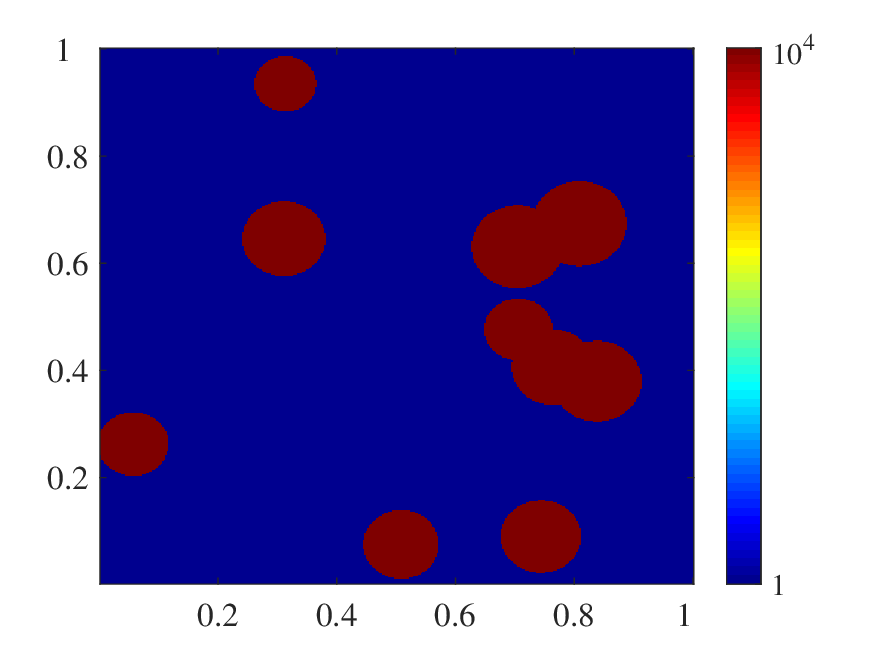}
  \end{minipage}}
	\subfigure[N = 20]{
		\label{exam2_over_coef2}
    \begin{minipage}[t]{0.3\linewidth}
      \centering
		\includegraphics[width=2in, height=1.9in]{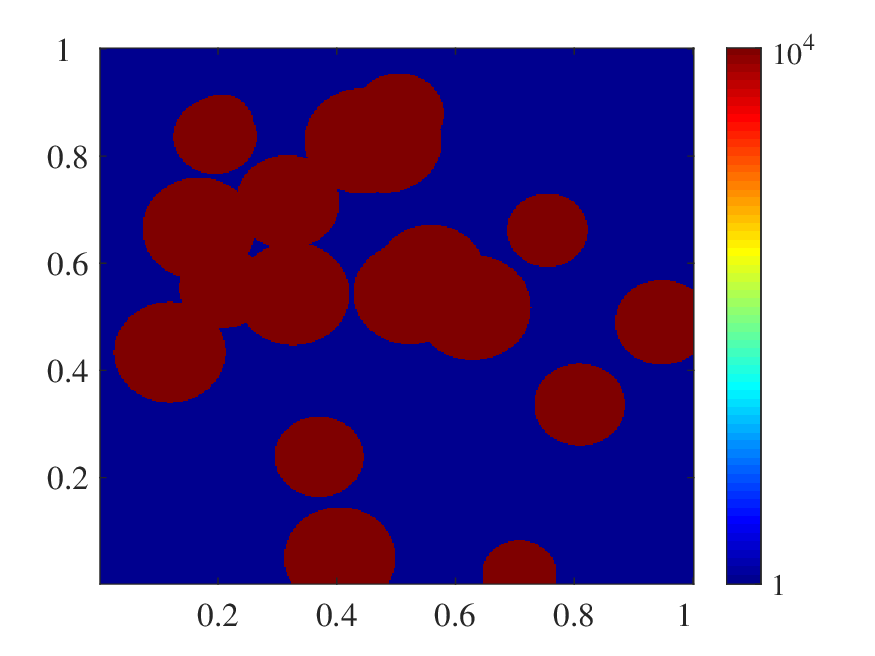}
  \end{minipage}}
	\subfigure[N =30]{
		\label{exam2_over_coef3}
    \begin{minipage}[t]{0.3\linewidth}
      \centering
		\includegraphics[width=2in, height=1.9in]{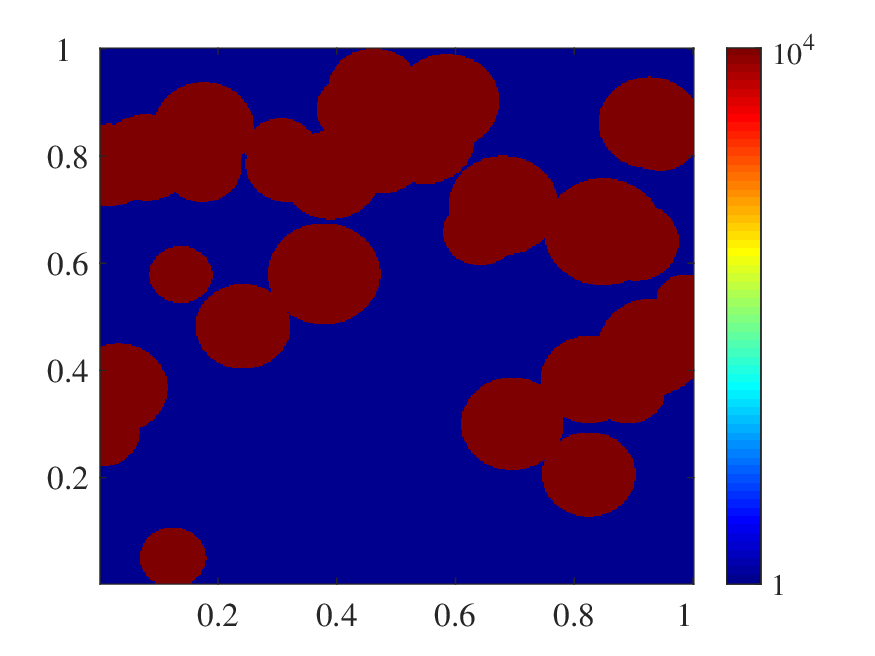}
  \end{minipage}}
  
	\subfigure[N = 40]{
		\label{exam2_over_coef4}
    \begin{minipage}[t]{0.3\linewidth}
      \centering
		\includegraphics[width=2in, height=1.9in]{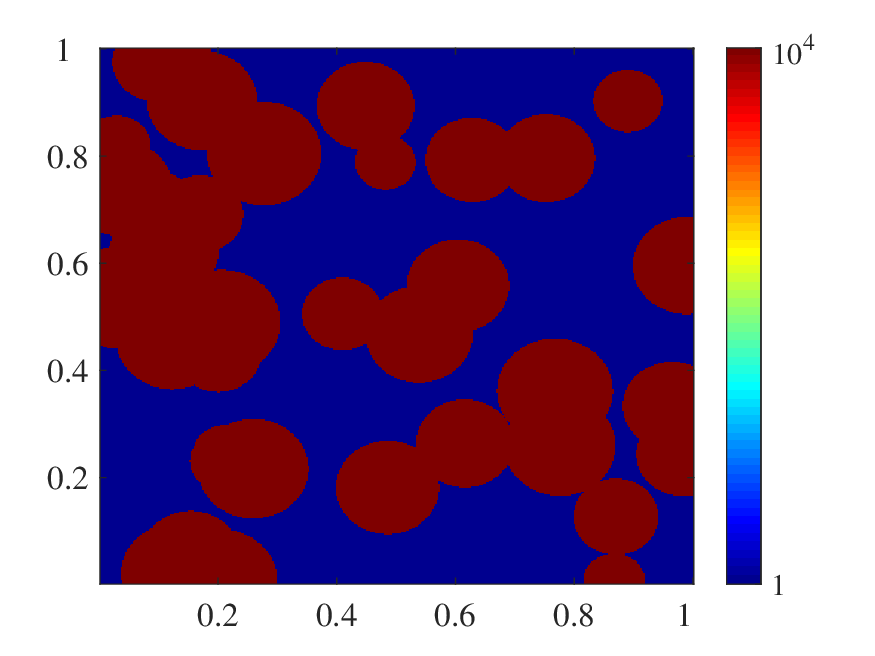}
  \end{minipage}}
	\subfigure[N = 60]{
		\label{exam2_over_coef5}
    \begin{minipage}[t]{0.3\linewidth}
      \centering
		\includegraphics[width=2in, height=1.9in]{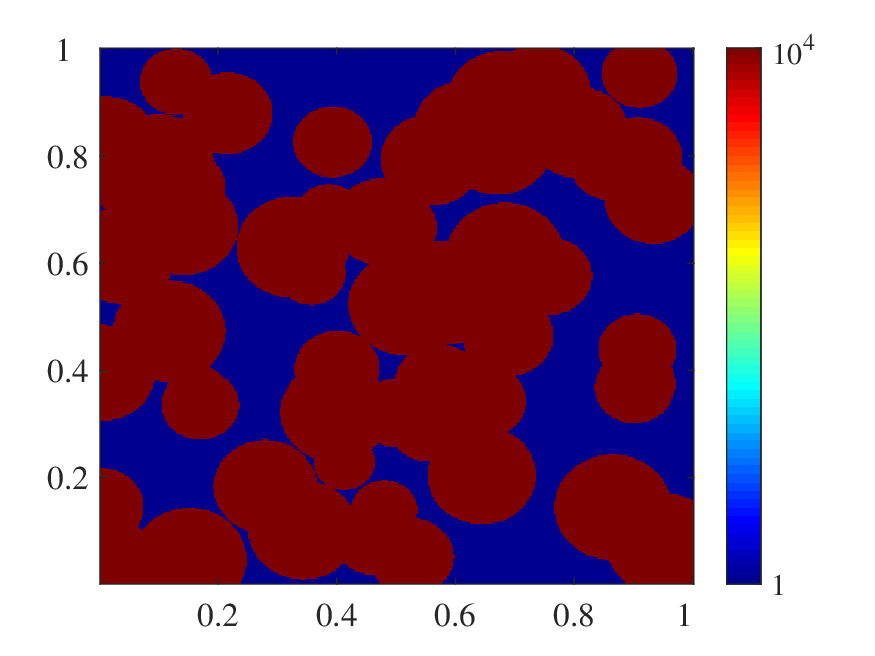}
  \end{minipage}}
	\subfigure[N = 80]{
		\label{exam2_over_coef6}
  \begin{minipage}[t]{0.3\linewidth}
    \centering
		\includegraphics[width=2in, height=1.9in]{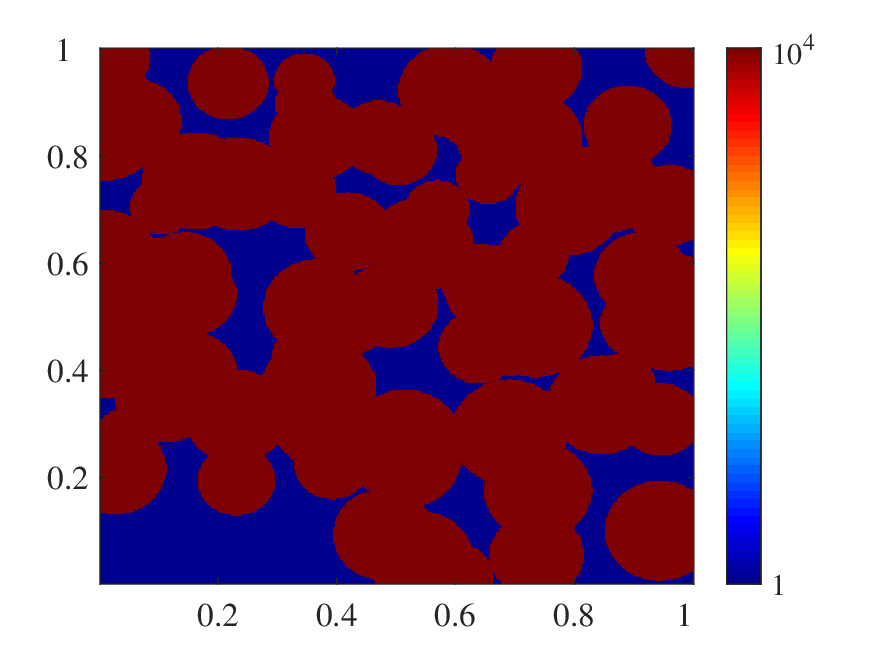}
  \end{minipage}}
	\caption{Map of field $\kappa(x)$ with various numbers of overlapping inclusions $N$.}
	\label{exam2_over_coef}
\end{figure}

\begin{table}[htbp]
	\caption{Comparison ($\varepsilon_{L_{2}}$ and $\varepsilon_{a}$) of various numbers of overlapping inclusions $N$.}
	\centering
	\begin{tabular}{|cc|c|c|c|c|l|l|}
		\hline
		\multicolumn{2}{|c|}{N}                         & \multicolumn{1}{l|}{10} & \multicolumn{1}{l|}{20} & \multicolumn{1}{l|}{30} & \multicolumn{1}{l|}{40} & 60        & 80        \\ \hline
		\multicolumn{1}{|c|}{\multirow{2}{*}{a=3}} & $\varepsilon_{L_{2}}$ & 1.5560e-3               & 3.1698e-3               & 7.8685e-3               & 8.2061e-3               & 7.0573e-2 & 5.8101e-1 \\ \cline{2-8} 
		\multicolumn{1}{|c|}{}                     & $\varepsilon_{a}$ & 1.9227e-2               & 3.2080e-2               & 5.1964e-2               & 5.3377e-2               & 2.1162e-1 & 7.1217e-1 \\ \hline
		\multicolumn{1}{|c|}{\multirow{2}{*}{a=6}} & $\varepsilon_{L_{2}}$ & 4.6618e-4               & 9.2879e-4               & 2.3236e-3               & 3.0367e-3               & 2.2186e-2 & 1.1618e-1 \\ \cline{2-8} 
		\multicolumn{1}{|c|}{}                     & $\varepsilon_{a}$ & 9.0417e-3               & 1.4037e-2               & 2.4599e-2               & 2.8893e-2               & 1.0796e-1 & 2.5358e-1 \\ \hline
	\end{tabular}
	\label{exam2_over_numbers}
\end{table}

Tabs. \ref{exam2_no_numbers} and \ref{exam2_over_numbers} show comparisons of the relative errors $\varepsilon_{L_{2}}$ and $\varepsilon_{a}$ of various numbers of overlapping or non-overlapping inclusions $N$. The number of basis functions $a$ in each coarse region is taken as 3 and 6. Consistent with the previous statement, more degrees of freedom lead to more accurate results. It is worth noting that accuracy drops significantly as the number of inclusions increases. This is because the higher the number of inclusions, the more complex the geometry of the coefficient $\kappa(x)$. This leads to a slower decay of the eigenvalue inverse of the corresponding eigenvalue problem. Combined with the previous numerical experiments, the errors of our presented method are independent of the contrast of the coefficients when the number of basis functions is  appropriately chosen, but are highly dependent on the geometry of the coefficients. How the geometry of the coefficients affects errors is beyond the scope of this paper, some simple conclusions can be found in \cite{efendiev2011multiscale}. For geometrically complex coefficients, we need more degrees of freedom to improve accuracy, as shown in the right side of Fig. \ref{example1-coef-basis}.
\begin{figure}[htpb]
	\subfigure[ ]{
		\begin{minipage}[t]{0.3\linewidth}
			\centering
			\includegraphics[width=2in, height=1.9in]{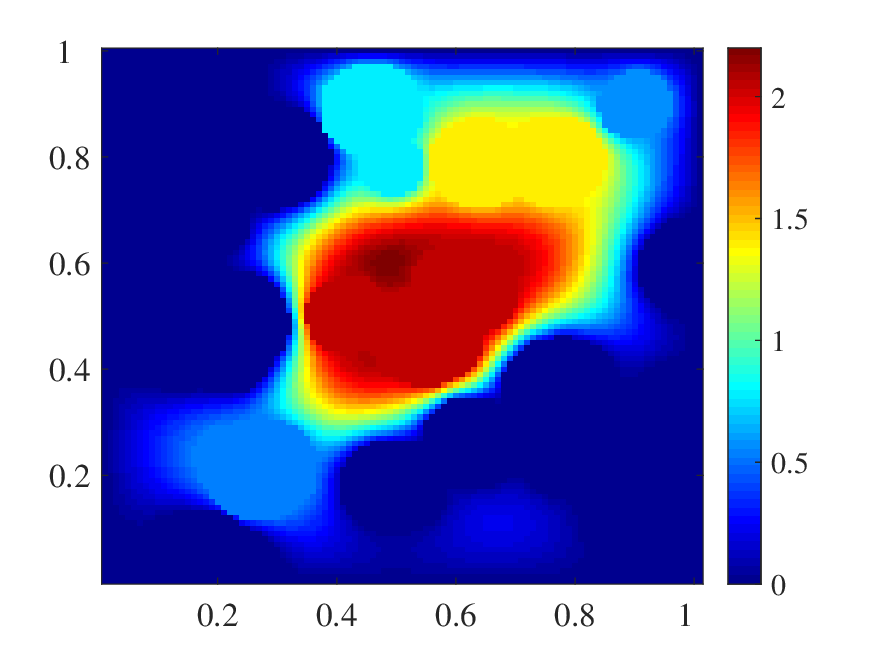}
		\end{minipage}
	}
	\subfigure[ ]{
		\begin{minipage}[t]{0.3\linewidth}
			\centering
			\includegraphics[width=2in, height=1.9in]{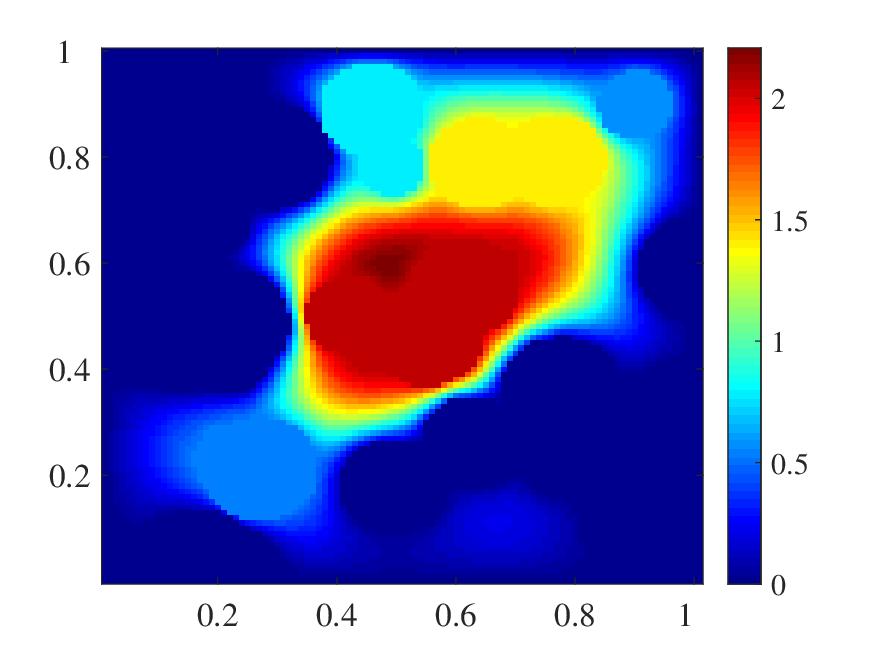}
		\end{minipage}
	}
	\subfigure[ ]{
		\begin{minipage}[t]{0.3\linewidth}
			\centering
			\includegraphics[width=2in, height=1.9in]{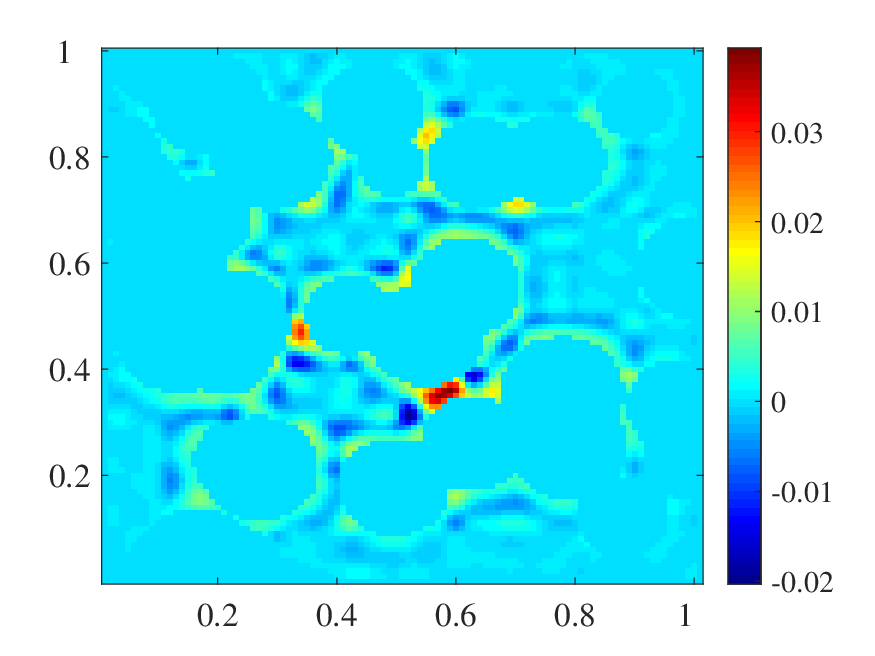}
		\end{minipage}
	}
	\caption{(a) the mean of reference solution at $T=0.1$ (b) the mean of CEM-GMsFEM solution at $T=0.1$ (c) the pointwise error between (a) and (b), where where $\kappa(x)$ is taken from Fig. \ref{example2_over_solutions}(d) and N = 40, M = 6.}
	\label{example2_over_solutions}  
\end{figure}

To be analogous with Fig. \ref{example1-solutions} in subsection \ref{Num-examp1}, Figs. \ref{example2_no_solutions} and \ref{example2_over_solutions} show the mean of the reference solution and the mean of the multiscale solution at $T = 0.1$ and the pointwise error between them, respectively. From these two figures, we observe that the CEM-GMsFEM solution mean is in excellent agreement with the reference solution mean. This illustrates the accuracy and effectiveness of our proposed methods.
\section{Conclusions}\label{Conclusions}
\qquad In this paper, a new semi-implicit stochastic multiscale method for the radiative heat transfer problems in composite materials is presented. The development of the proposed method requires temporal discrete numerical schemes, the discrete representation of infinite-dimensional noise and multiscale basis construction. For each fixed time step, we employ a semi-implicit numerical scheme to linearize the resulting equations and apply two-step predictor-corrector methods to estimate strong nonlinearity terms induced by heat radiation more accurately. The representation of infinite-dimension additive noises as abstract forms is presented to reduce the dimension of the random space, and Monte Carlo sampling is used for probability space. Our coarse space consists of CEM-GMsFEM basis functions. Through the convergence analysis, it is concluded that our proposed methods not only retain the advantages of CEM-GMsFEM, which is linearly dependent on coarse grids and independent of high contrast coefficients, but can also be used to characterize the influence of noise fluctuation on the solutions. To illustrate the applicability of the proposed method, we presented numerical results for two test problems. Numerical results showed that our proposed method has good universality. It is not limited to periodic microstructure and can accurately calculate the nonlinear radiative heat transfer problems of various multiscale structures. This proposed method will be generalized for the simulation of time-dependent nonlinear stochastic multiscale problems.
\begin{acknowledgements}
  The authors gratefully appreciate the valuable comments from the reviewers, which have contributed significantly to the improvement of this manuscript. This research was supported by National Natural Science Foundation of China (11301392), and Fundamental Research Funds for Central Universities.
  
  \end{acknowledgements}

  \begin{dataavailabilitys}
    Datasets produced or examined in the present study can be obtained from the corresponding author upon reasonable request.
    \end{dataavailabilitys}


%
\section*{Conflict of interest}
The authors declare that they have no conflict of interest.

\bibliographystyle{spmpsci}      
\bibliography{heatradiation.bib}  

%
%

\end{document}